\documentclass[11pt,reqno]{amsart}
\usepackage{amsmath,amssymb,amsthm,amsfonts,mathrsfs,latexsym,times,color,bm}
\usepackage{amsmath,amsfonts}
\usepackage{amsthm}
\usepackage{graphicx}
\usepackage{color}
\usepackage{multicol}
\usepackage{hyperref}

\newtheorem{thm}{Theorem}[section]

\newtheorem{Lemma}[thm]{Lemma}
\newtheorem{remark}{Remark}[section]
\newtheorem{theorem}[thm]{Theorem}
\newtheorem{proposition}[thm]{Proposition}

\numberwithin{equation}{section}

\newcommand{\beq}{\begin{equation}}
\newcommand{\eeq}{\end{equation}}
\newcommand{\ben}{\begin{eqnarray}}
\newcommand{\een}{\end{eqnarray}}
\newcommand{\beno}{\begin{eqnarray*}}
\newcommand{\eeno}{\end{eqnarray*}}

\newcommand{\bb}{\mathbf{b}}
\newcommand{\bv}{\mathbf{v}}

\newcommand{\bh}{\mathbf{h}}

\newcommand{\bF}{\mathbf{F}}
\newcommand{\bJ}{\mathbf{J}}
\newcommand{\bG}{\mathbf{G}}
\newcommand{\bE}{\mathbf{E}}

\newcommand{\bB}{\mathbf{B}}
\newcommand{\bH}{\mathbf{H}}

\newcommand{\by}{\mathbf{y}}
\newcommand{\bx}{\mathbf{x}}

\numberwithin{equation}{section}

\begin{document}
\title[Ideal incompressible MHD equations]{Low regularity solutions for the Cauchy problem of the ideal incompressible Magnetohydrodynamics equations}


\author{Huali Zhang}
\address{Hunan University, Lushan South Road in Yuelu District, Changsha, 410882, People's Republic of China.}

\email{hualizhang@hnu.edu.cn}

\date{\today}


\begin{abstract}
	In Lagrangian coordinates, the local well-posedness of low regularity solutions is established for an ideal incompressible magnetohydrodynamic (MHD) system subject to a homogeneous background magnetic field. First, the MHD system is reformulated into a degenerate wave-elliptic system with a particular null structure. By introducing a suitably defined solution space, several refined product estimates are derived. Next, using the inherent null structure, a Klainerman-Machedon type bilinear estimate is obtained for the nonlinear terms. These nice structures and estimates yield the local well-posedness of the ideal incompressible MHD equations in Lagrangian coordinates for initial velocity fields $\bv_0 \in H^{s}(\mathbb{R}^n)$ with $s > \frac{n+1}{2}$ $(n=2,3,4)$. {Remarkably}, the regularity requirement is lowered by half a derivative compared with the classical exponent $s > \frac{n}{2}+1$.

\end{abstract}

\maketitle
\section{Introduction}
Magnetohydrodynamics (MHD) studies the dynamic interplay between velocity and magnetic fields within electrically conducting fluids, including plasmas, liquid metals, seawater, and electrolytes. This paper focuses on establishing the well-posedness of low regularity solutions to the Cauchy problem for the ideal MHD system expressed in Lagrangian coordinates. To introduce the problem, we first present the incompressible MHD system in its Eulerian formulation, given by
\begin{equation}\label{MHD}
	\begin{cases}
		& \partial_t \bar{\bv}  + (\bar{\bv} \cdot \nabla) \bar{\bv}- (\bar{\bb} \cdot \nabla) \bar{\bb}+\nabla \bar{q}=0,\quad (t,\bx)\in \mathbb{R}^+\times \mathbb{R}^n,
		\\
		& \partial_t \bar{\bb} + (\bar{\bv} \cdot \nabla) \bar{\bb} - (\bar{\bb} \cdot \nabla) \bar{\bv} =0,\qquad \quad \ \ (t,\bx)\in \mathbb{R}^+\times \mathbb{R}^n,
		\\
		& \text{div}\bar{\bv}=\text{div}\bar{\bb}=0, \qquad \qquad \qquad \qquad \quad \ \ (t,\bx)\in \mathbb{R}^+\times \mathbb{R}^n,
		\\
		& (\bar{\bv}, \bar{\bb})|_{t=0}=(\bar{\bv}_0, \bar{\bb}_0),  \qquad  \qquad \qquad  \qquad \quad  \bx\in \mathbb{R}^n,
	\end{cases}
\end{equation}
where $\bar{\bv}=(\bar{v}^1,\bar{v}^2,\cdots,\bar{v}^n)^{\mathrm{T}}$ denotes the velocity field, $\bar{\bb}=(\bar{b}^1,\bar{b}^2,\cdots,\bar{b}^n)^{\mathrm{T}}$ the magnetic field, and $\bar{q}$ the pressure. $\bar{\bv}_0$ and $\bar{\bb}_0$ are the initial data which satisfy the divergence free condition
\begin{equation*}
	\mathrm{div}\bar{\bv}_0=\mathrm{div} \bar{\bb}_0=0, \quad \bx\in \mathbb{R}^n.
\end{equation*}
Next, we introduce the Lagrangian trajectory $\bx(t,\by)=(x^1,x^2,\cdots,x^n)^{\text{T}}$ by
\begin{equation}\label{Lag2}
	\begin{cases}
		&\frac{d}{dt} x^i(t,\by)=\bar{v}^i(t,\bx(t,\by)),
		\\
		&x^i(0,\by)=y^i.
	\end{cases}
\end{equation} 
Define
\begin{equation}\label{Lag3}
	\begin{split}
		& b^i(t,\by)=\bar{b}^i(t,\bx(t,\by)),\qquad  v^i(t,\by)=\bar{v}^i(t,\bx(t,\by)), \qquad q(t,\by)=\bar{q}(t,\bx(t,\by)),
	\end{split}
\end{equation} 
and
\begin{equation}\label{Lag4}
	\begin{split}
		H^{ia}=\frac{\partial x^i}{\partial y^a}.
	\end{split}
\end{equation} 
For smooth solutions, when \(\bb_0\) is a non-zero constant vector, system \eqref{MHD}  can be reformulated as follows (see Lemma \ref{Lagm} for details)  
\begin{equation}\label{Lag5}
	\begin{cases}
		&\frac{\partial^2 H^{ia}}{\partial t^2}-b_0^k b_0^j \frac{\partial^2 H^{ia}}{\partial y^k y^j}+ \frac{\partial}{\partial y^a} (\partial_{x_i} q)=0,
		\\
		& H^{ia}|_{t=0}=\delta^{ia}, \quad \frac{\partial H^{ia}}{\partial t}|_{t=0}=\frac{\partial v^i_0}{\partial y^a},
	\end{cases}
\end{equation}
where $q$ and $\bH=(H^{ia})_{n\times n}$ satisfy\footnote{Here $\bH^{-1}$ is a inverse matrix of $\bH$. The operator $\Delta_x$ is denoted by $\Delta_x=\sum^n_{i=1}\frac{\partial^2 }{ \partial x^i \partial x^i}$. }
\begin{align}\label{Lag5A}
	\Delta_x {q}
	=& (\bH^{-1})^{ki} (\bH^{-1})^{mj} \big(  \frac{\partial H^{jk} }{\partial t}   \frac{\partial H^{im}}{\partial t}  
	-b_0^l b_0^r \frac{\partial H^{jk}}{\partial y^l}  \frac{\partial H^{im}}{\partial y^r} \big),
	\\\label{Lag5B}
	(\bH^{-1})^{ia} = & \frac{1}{(n-1)!} \epsilon_{i i_2 i_3 \cdots i_n }\epsilon^{a a_2 a_3 \cdots a_n } H^{i_2 a_2}  \cdots H^{i_n a_n}.
\end{align}
Let
\begin{equation*}
	\bB=
	\left(
	\begin{array}{cccc}
		b_0^1 b_0^1 & b_0^1 b_0^2 & \cdots & b_0^1 b_0^n  \\
		b_0^1 b_0^2 & b_0^2 b_0^2 & \cdots & b_0^2 b_0^n \\
		\cdots & \cdots & \cdots & \cdots \\
		b_0^1 b_0^n & b_0^2 b_0^n & \cdots & b_0^n b_0^n
	\end{array}
	\right )
\end{equation*}
denote the coefficient matrix of the second term in equation \eqref{Lag5}. When $\bb_0 \neq \mathbf{0}$, $\bB$ is a positive semi-definite matrix of rank one. Consequently, the coupled system \eqref{Lag5}--\eqref{Lag5B} forms a degenerate wave--elliptic system in $n \geq 2$ spatial dimensions. In this work, we investigate the local well-posedness of this system under very low regularity assumptions on the initial data. Before presenting our main results, we briefly recall relevant prior contributions in the literature.

\subsection{Previous results}
In the case of $\bb=0$, then \eqref{MHD} is reduced to the classical incompressible Euler equations 
\begin{equation}\label{IEE}
	\begin{cases}
		& \partial_t \bar{\bv} + (\bar{\bv} \cdot \nabla) \bar{\bv}+\nabla \bar{q}=0,
		\\
		& \text{div}\bar{\bv}=0, 
		\\
		& \bar{\bv}|_{t=0}=\bar{\bv}_0. 
	\end{cases}
\end{equation}
Kato and Ponce \cite{KP} proved the local well-posedness of \eqref{IEE} when the initial velocity $\bar{\mathbf{v}}_0$ belongs to the Sobolev space $W^{s,p}(\mathbb{R}^n)$ with $s > 1 + \frac{n}{p}$. Chae \cite{Chae} established the local existence of solutions by considering initial data $\bar{\mathbf{v}}_0$ in Triebel-Lizorkin spaces. Conversely, Bourgain and Li \cite{BL, BL2} demonstrated that the Cauchy problem is ill-posed for initial velocities $\bar{\mathbf{v}}_0 \in W^{1 + \frac{n}{p}}(\mathbb{R}^n)$, where $1 \leq p < \infty$ and $n = 2,3$. Recently, Kim and Jeong \cite{KJ} proved the ill-posedness of \eqref{IEE} when $\bar{\bv}_0 \in H^{\frac{n}{2}+1}(\mathbb{R}^n)$, $n\geq 3$.

Interestingly, the well-posedness theory differs in the context of elastic materials. For incompressible Neo-Hookean elastic equations, Andersson and Kapitanski \cite{AK} derived a first-order half-wave system for the vorticity in Lagrangian coordinates and established Strichartz estimates for the velocity and the elastic tensor. Under additional H\"older regularity assumptions on the vorticity, they proved well-posedness for low-regularity solutions in both Lagrangian and Eulerian coordinates, requiring $s > \frac{7}{4}$ for $n=2$ or $s > 2$ for $n=3$. By finding a ``wave map" null structure embedded in the equations, Zhang \cite{Zhang} obtained a refined well-posedness result for incompressible Neo-Hookean elastic equations. In the Eulerian framework, the regularity requirement remains $s > \frac{7}{4}$ for $n=2$ or $s > 2$ for $n=3$. For general initial data in Lagrangian coordinates, the condition relaxes to $s > \frac{n+1}{2}$ for $n=2,3$; and for small initial data in Lagrangian coordinates, it further reduces to $s > \frac{n}{2}$ for $n=2,3$. Similarly, it is also of great interest to study the low regularity problem for the ideal incompressible MHD system.

The general form of the MHD system is described by the equations
\begin{equation}\label{MHD0}
	\begin{cases}
		& \partial_t \bar{\bv} -\nu \Delta \bar{\bv} + (\bar{\bv} \cdot \nabla) \bar{\bv}- (\bar{\bb} \cdot \nabla) \bar{\bb}+\nabla \bar{q}=0,\quad (t,\bx)\in \mathbb{R}^+\times \mathbb{R}^n,
		\\
		& \partial_t \bar{\bb} -\mu \Delta \bar{\bb} + (\bar{\bv} \cdot \nabla) \bar{\bb} - (\bar{\bb} \cdot \nabla) \bar{\bv} =0,\qquad \quad \ \ (t,\bx)\in \mathbb{R}^+\times \mathbb{R}^n,
		\\
		& \text{div}\bar{\bv}=\text{div}\bar{\bb}=0, \qquad \qquad \qquad \qquad \qquad \qquad \ \ \ (t,\bx)\in \mathbb{R}^+\times \mathbb{R}^n,
		\\
		& (\bar{\bv}, \bar{\bb})|_{t=0}=(\bar{\bv}_0, \bar{\bb}_0),  \qquad  \qquad \qquad  \qquad \qquad \qquad  \bx\in \mathbb{R}^n,
	\end{cases}
\end{equation}
where the viscosity $\nu\geq 0$ and magnetic diffusivity $\mu\geq 0$ are constants. There is an extensive literature on the progress of local and global well-posedness problems for \eqref{MHD0}. In the fully dissipative case $\nu > 0$ and $\mu > 0$, it is well known
that the classical solution in two dimensions is global in time, and the weak solution is regular and unique. However, the global existence of classical solutions and the regularity of weak solutions remain challenging open problems for the three-dimensional MHD system. For more details, please refer to Sermange and Temam's paper \cite{ST}. In the non-dissipative but resistive case $\nu = 0,\ \mu > 0$, Lei and Zhou \cite{LZw} proved the global existence of weak solutions for the two-dimensional system. Furthermore, under mixed partial regularity assumptions, global well-posedness can also be obtained by Cao and Wu \cite{CW}.

In the non-resistive case $\mu = 0,\ \nu > 0$, Jiu and Niu \cite{JN} established local existence in 2D for initial data in $H^s$ with integer $s \geq 3$. Later, Fefferman et al. \cite{Feff1} proved local well-posedness for $(\bb_0,\bv_0) \in H^s(\mathbb{R}^n)$ with $s > \frac{n}{2}$. By deriving a new maximal regularity property of the heat equation, Fefferman et al. \cite{Feff} further lowered the regularity requirement to $(\bb_0,\bv_0) \in H^s(\mathbb{R}^n) \times H^{s-1+\varepsilon}(\mathbb{R}^n)$ for $s > \frac{n}{2}$ and any $0 < \varepsilon < 1$. Recently, Chen-Nie-Ye \cite{CNY} showed that the problem is ill-posed when $(\bb_0,\bv_0) \in H^{n/2}(\mathbb{R}^n) \times H^{n/2-1}(\mathbb{R}^n)$ for $n \geq 2$, thus settling the sharp regularity threshold in Sobolev spaces for the incompressible non-resistive MHD system. When the problem is considered in Besov spaces, the situation changes slightly. Chemin et al. \cite{Chemin} obtained local existence for $(\bb_0,\bv_0) \in B^{n/2}_{2,1}(\mathbb{R}^n) \times B^{n/2-1}_{2,1}(\mathbb{R}^n)$ with $n = 2,3$. Subsequently, Li-Tan-Yin \cite{LTY} and Li-Yin-Zhu \cite{LYZ} established well-posedness under the condition $(\bb_0,\bv_0) \in \dot{B}^{n/p}_{p,1}(\mathbb{R}^n) \times \dot{B}^{n/p-1}_{p,1}(\mathbb{R}^n)$ for $1 \leq p \leq 2n$ and $n \geq 2$. Regarding global solutions, especially for small smooth initial data, we refer to the insightful result of Lin and Zhang \cite{LZ}. Motivated by \cite{LZ}, a number of interesting works have been devoted to the global well-posedness of the non-resistive incompressible MHD system; see, for example, \cite{LXZ} and \cite{XZ}.

In the case $\mu = \nu = 0$, system \eqref{MHD0} reduces to the ideal incompressible MHD equations \eqref{MHD}. The global well-posedness of strong small solutions has been investigated by many authors in various settings; see, for example, Bardos-Sulem-Sulem \cite{BB}, He-Xu-Yu \cite{HXY}, Cai-Lei \cite{CL}, and Wei-Zhang \cite{WZ}. Recently, Faraco et al. \cite{Fa} studied weak solutions of the ideal MHD system. For local strong solutions, Schmidt \cite{Schmidt} and Secchi \cite{Sec} established local existence when the initial data belong to $H^s(\mathbb{R}^n)$ with $s > 1 + \frac{n}{2}$. Chen, Miao, and Zhang \cite{CMZ} proved local well-posedness in Triebel-Lizorkin spaces, while Miao and Yuan \cite{MY} obtained corresponding results in Besov spaces. Very recently, Ifrim, Pineau, Tataru, and Taylor \cite{IBTT} established sharp regularity results for the free boundary problem of the ideal MHD system in Eulerian coordinates.
\subsection{Motivation}
Despite recent advances in higher regularity settings and related free boundary problem, the Cauchy problem for the ideal MHD system remains unsolved at low regularity.

The basic energy estimates of system \eqref{MHD} gives
\begin{equation*}
	\|(\bv,\bb)\|_{L^\infty_t H^s} \lesssim \|(\bv_0,\bb_0)\|_{H^s} \exp\left( \int^t_0 \| (\nabla\bv, \nabla\bb)\|_{L^\infty_x} d\tau \right).
\end{equation*}
Therefore, existing well-posedness results \cite{CMZ,MY,Schmidt,Sec} rely on the Sobolev imbedding $H^{\frac{n}{2}+} \hookrightarrow  L^\infty$, which permits the estimate
\begin{equation*}
	\|(\nabla\bv, \nabla\bb)\|_{L^\infty_x} \lesssim \|(\bv,\bb)\|_{H^s},\quad s>\frac{n}{2}+1.
\end{equation*}
However, for regularity indices $s \leq \frac{n}{2}+1$, the quantities $\|\nabla\bv\|_{L^\infty_x}$ and $\|\nabla\bb\|_{L^\infty_x}$ can no longer be bounded by the $H^s$-norm of the solution. This loss of pointwise control over gradients introduces fundamental analytical difficulties in closing energy estimates, leaving the well-posedness theory in this range an open problem.

From \cite{BB,CL,HXY}, it is known that the ideal MHD system \eqref{MHD} can be reformulated as a one-dimensional quasilinear wave system in Eulerian coordinates. In the context of low regularity solutions for quasilinear wave equations, Strichartz estimates are a fundamental tool (see the references \cite{KR2} and \cite{SmT} for details). A natural approach would be to bound a quantity such as $\|(\nabla\bv, \nabla\bb)\|_{L_t^1L^\infty_x}$, which has been successfully carried out for related systems such as the incompressible neo-Hookean elastic equations \cite{AK, Zhang}. To the best of our knowledge, no Strichartz estimates are available for the one-dimensional wave equation, which creates a major trouble for studying low regularity solutions of \eqref{MHD} in Eulerian coordinates.


In our paper, a crucial observation is that when the ideal MHD system \eqref{MHD} is expressed in Lagrangian coordinates, it reduces to a one-dimensional semilinear wave system with a specific null structure--precisely, equations \eqref{Lag5}--\eqref{Lag5B}. In this reformulation, the principal difficulty in the Eulerian setting--the need to control $\|(\nabla\bv, \nabla\bb)\|_{L^1_tL^\infty_x}$ is circumvented. Indeed, we introduce a carefully chosen solution space $H^{s-1,1}_{\theta}$ (see \eqref{HS1} and \eqref{HS2} below), defined through a norm that respects both the one-dimensional wave operator and the null structure of the nonlinearities. Using the reformulated system \eqref{Lag5}--\eqref{Lag5B} and the variable substitution\footnote{The matrix $\bH=(\frac{\partial \bx}{\partial \by})_{n\times n}$, and $\bE$ is an identity matrix} $\bG=\bH-\bE$, we can derive a new energy estimate (see Section \ref{PT} for details)
\begin{equation*}
	\begin{split}
		| \bG |_{H^{s-1,1}_{\theta}} \lesssim \|\bv_0\|_{H^s}+ T^{\frac{\varepsilon}{2}}(1+ | \bG |^{n^2+3n+1}_{H^{s-1,1}_{\theta}}) \|Q_0(\bG,\bG)\|_{H^{s-1,0}_0},
	\end{split}
\end{equation*}
and prove the bilinear estimate (see Lemma \ref{eQR} for details)
\begin{equation*}
\|Q_0(\bG,\bG)\|_{H^{s-1,0}_0} \lesssim | \bG |^2_{H^{s-1,1}_{\theta}}.
\end{equation*}
The above nice structure and estimates strongly motivate the study of low regularity well-posedness for the ideal incompressible MHD system in Lagrangian coordinates. Specifically, we establish the local well-posedness of low regularity solutions to \eqref{Lag5}--\eqref{Lag5B} under the assumptions $\mathbf{v}_0 \in H^s(\mathbb{R}^n)$ with $s > \frac{n+1}{2}$, for $n = 2, 3, 4$, and $\mathbf{b}_0 $ being a nonzero constant vector. This provides the first low-regularity well-posedness result for the reduced Lagrangian system \eqref{Lag5}--\eqref{Lag5B} in this regularity regime below the classical exponent $s>\frac{n}{2}+1$, and serves as a key step toward identifying the sharp regularity required for the Cauchy problem. Our main results are stated as follows.



\subsection{Statement of the result}
\begin{theorem}\label{thm}
	Assume that $\frac{n+1}{2}<s\leq \frac{n}{2}+1$ and ${2 \leq n \leq 4 } $. Suppose that $\bv_0\in H^s(\mathbb R^n)$ and that $\bb_0$ is a nonzero constant vector. Then there exists $T>0$ depending only on $n, s, |\bb_0|$, and $\| \bv_0\|_{H^s}$ such that the system \eqref{Lag5}--\eqref{Lag5B} admits a unique solution with
	\begin{equation*}
		\bH-\bE \in C([0,T],H^{s-1,1}(\mathbb{R}^n))\cap C^1([0,T],H^{s-1,0}(\mathbb{R}^n)),
	\end{equation*}
	 where $\bE$ is a $n\times n$ identity matrix. Moreover, the solution map is continuous from \(H^s(\mathbb R^n)\) to
	  $$C([0,T];H^{s-1,1}(\mathbb{R}^n))\cap C^1([0,T];H^{s-1,0}(\mathbb{R}^n)).$$
\end{theorem}

\begin{remark}
	\begin{enumerate}
	\item	Compared to the classical requirement $ s > \frac{n}{2} + 1 $ established by Schmidt \cite{Schmidt} and Secchi \cite{Sec}, Theorem \ref{thm} lowers the regularity demand on the initial velocity by half a derivative. It is also interesting to compare our result with the incompressible Euler equations, for which Bourgain and Li \cite{BL, BL2} showed that the condition $ s > \frac{n}{2} + 1 $ ($n=2,3$) is in fact sharp.
	\item
	 For the free boundary problem of the incompressible ideal MHD system, well-posedness holds when the regularity index satisfies $ s > \frac{n}{2} + 1 $, as shown in \cite{IBTT}. By contrast, for the Cauchy problem, the regularity threshold can be lowered by establishing bilinear estimates on the solutions. Such improvements may not be obtained for free boundary problems.
  \end{enumerate}
\end{remark}

\begin{remark}
	\begin{enumerate}
		\item
	Our work is further motivated by the studies of Andersson-Kapitanski \cite{AK} and Zhang \cite{Zhang} on the incompressible neo-Hookean elastic equations. We observe that the solution spaces introduced in \cite{AK,Zhang} are not directly applicable to the wave component in \eqref{Lag5}--\eqref{Lag5B}, since the latter involves a degenerate one-dimensional operator. To our knowledge, no Strichartz estimates are available for one-dimensional wave equations. Consequently, if we wish to lower the regularity requirements on the initial data, the problem should be approached through bilinear estimates rather than Strichartz estimates.
	\item
	The first key challenge is therefore to define a new solution space that is adapted to the structure of system \eqref{Lag5}. A second essential point is to verify that, within this new function framework, the relevant product estimates and bilinear estimates for the null form can be established. This strategy draws inspiration from the works of Klainerman and Machedon \cite{KM1}, Foschi and Klainerman \cite{FK}, Klainerman and Selberg \cite{KS2} on nonlinear wave equations.
\end{enumerate}
\end{remark}

\begin{remark}
\begin{enumerate}
		\item The assumption in Theorem \ref{thm} that $\bb_0$ is a constant vector plays a key role, since the proof relies on the solution representation for linear wave equations with constant coefficients as well as the Fourier transform analysis of the associated null form. For details we refer to Proposition \ref{nE} and Lemma \ref{eQR}. A natural follow-up question is whether low regularity well-posedness remains valid when $\bb_0$ is asymptotically flat or is a small perturbation of a nonzero constant vector.
	 \item In Section \ref{PT}, we require the transformation between norms in Eulerian and Lagrangian coordinates. By Lemma \ref{LD4}, the regularity exponent $s$ must satisfy either $s - 1 \in (0,1]$ or $s - 2 \in (0,1]$. Because we require $s > \frac{n+1}{2}$, this forces the spatial dimension to satisfy $2 \leq n \leq 4$.
\end{enumerate}
\end{remark}

\begin{remark}
 It is important to note that Theorem \ref{thm} holds only in Lagrangian coordinates. Moreover, it is not clear whether this low-regularity Lagrangian solution is equivalent to an Eulerian solution in the corresponding Sobolev spaces. Consequently, Theorem \ref{thm} does not automatically imply local well-posedness for \eqref{MHD} in $H^s$ space.  
\end{remark}

\subsection{Notations}
If $f$ and $g$ are two functions, we say $f \lesssim g$ if and only if there exists a constant
$C>0$ such that $f\leq C g$. We say $f\approx g$ if and only if there exist constants $C_1, C_2>0$ such that $ C_1 f \leq g \leq C_2 f$. The constant $C$ may change from line to line.

Space Fourier transforms on $\mathbb{R}^{n}$ are denoted by $\widehat{\cdot}$ :
\begin{equation*}
\widehat{f}(\xi)=\int_{\mathbb{R}^n} \text{e}^{\text{i}\bx\cdot \xi}f(\bx)d\bx, 
\end{equation*}
and space-time Fourier transforms on $\mathbb{R}^{1+n}$ are denoted by $\widetilde{\cdot}$ :
\begin{equation*}
\widetilde{F}(\tau,\xi)=\int_\mathbb{R} \int_{\mathbb{R}^n} \text{e}^{\text{i}(t\tau+\bx\cdot \xi)}F(t,\bx)d\bx dt.
\end{equation*}
When no confusion can arise, we also use the Fourier transform on $\mathbb{R}$ by
	$\widehat{w}(\tau)=\int_{\mathbb{R}} \text{e}^{\text{i}t \tau}w(t)dt$.
For $a, b, \theta \in \mathbb{R}$, denote the space $H^{a,b}_{\theta}$ by\footnote{The space $\mathcal{S}'(\mathbb{R}^{1+n})$ is the dual space of Schwartz functions.}
\begin{equation*}
	H^{a,b}_{\theta}=\left\{ u\in \mathcal{S}'(\mathbb{R}^{1+n}): \left< \xi \right>^{a} \left< \xi_1 \right>^{b} \left<||\tau|-|\xi_1||\right>^{\theta} \widetilde{u}(\tau,\xi) \in L^2(\mathbb{R}^{1+n}) \right\},
\end{equation*}
where $\xi=(\xi_1,\xi_2,\cdots,\xi_n)^{\mathrm{T}},$ $\left< \xi \right>=1+|\xi|, \left< \xi_1 \right>=1+|\xi_1|$ and  $\left<|\tau|-|\xi_1|\right>=1+ \left||\tau|-|\xi_1| \right|$. We use the notation $\|f\|_{s,\theta}$ to denote a norm in $H^{s,\theta}$, that is
\begin{equation}\label{HS1}
	\|u\|_{H^{a,b}_{\theta}}=\| \left< \xi \right>^a \left< \xi_1 \right>^{b} \left<||\tau|-|\xi_1||\right>^{\theta} \widetilde{u}(\tau,\xi) \|_{ L^2(\mathbb{R}^{1+n}) }.
\end{equation}
We also introduce a norm
\begin{equation}\label{HS2}
	|f|_{H^{a,b}_{\theta}}=\|f\|_{H^{a,b}_{\theta}}+ \|\partial_t f\|_{H^{a,b-1}_{\theta}}.
\end{equation}

The operators $\Lambda$, $\Lambda_{-}$, $\Lambda_1$ and $D$ are denoted by
\begin{equation*}\label{Na}
	\begin{split}
		&\widehat{\Lambda^\alpha f}(\xi)=\left< \xi \right>^\alpha \widehat{f}(\xi),
	\qquad \ \widehat{\Lambda_1^\alpha f}(\xi)=\left< \xi_1 \right>^\alpha \widehat{f}(\xi),
		\\
		& \widehat{D^\alpha f}(\xi)=| \xi_1 |^\alpha \widehat{f}(\xi),
	\quad
		\widehat{\Lambda_{-}^\alpha F}(\tau,\xi)=\left<||\tau|-|\xi_1||\right>^\alpha \widetilde{F}(\tau,\xi).
	\end{split}
\end{equation*}
Denote the operator $\widetilde{\square}$ by
\begin{equation}\label{wo}
	\widetilde{\square}=\frac{\partial^2}{\partial t^2}-\frac{\partial^2}{\partial y_1^2}.
\end{equation}
For two functions $f$ and $g$, we also define $Q_0$ by
\begin{equation}\label{nuf}
	Q_0(f,g)= \partial_t f \partial_t g - \partial_{y_1} f \partial_{y_1} g.
\end{equation} 
For $a,b \in \mathbb{R}$, we define the space
\begin{equation*}
	\begin{split}
			H^{a}(\mathbb{R}^n)=& \{ u\in \mathcal{S}'(\mathbb{R}^{n}): \left< \xi \right>^a  \widehat{u}(\xi) \in L^2(\mathbb{R}^{n}) \}, 
	\end{split}
\end{equation*}
and
\begin{equation*}
	\begin{split}
		H^{a,b}(\mathbb{R}^n)=& \{ u\in \mathcal{S}'(\mathbb{R}^{n}): \left< \xi \right>^a \left< \xi_1 \right>^b  \widehat{u}(\xi) \in L^2(\mathbb{R}^{n}) \}.
	\end{split}
\end{equation*}
Consequently, $H^{a,0}(\mathbb{R}^n)=H^{a}(\mathbb{R}^n)$. Introduce two cut-off functions $\chi$ and $\phi$ respectively satisfying
 \begin{equation}\label{chi}
	\chi \in C^\infty_c(\mathbb{R}), \quad \chi=1 \ \text{on} \ [-1,1], \quad \text{supp} \chi \subseteq (-2,2).
 \end{equation}
 and
 \begin{equation}\label{phi}
	\phi \in C^\infty_c(\mathbb{R}), \quad \phi=1 \ \text{on} \ [-2,2], \quad \text{supp} \phi \subseteq (-4,4).
 \end{equation}
\subsection{Organization of the paper}
In the next section, \ref{app}, we derive the wave-elliptic formulation of \eqref{MHD} in the Lagrangian framework. In Section \ref{pr}, we establish several inequalities and product estimates. In Section \ref{be}, we prove bilinear estimates for null forms. Section \ref{PT} presents the proof of Theorem \ref{thm}. Finally, in the appendix, Section \ref{App}, we provide the proof of Proposition \ref{nE}.

\section{Derivation of incompressible MHD equations in Lagrangian coordinates}\label{app}
\begin{Lemma}\label{Lagm}
	Let $(\bar{\bb}, \bar{\bv})$ be a smooth solution of system \eqref{MHD}, with the initial magnetic field taken as a constant vector. Using the Lagrangian trajectory defined in \eqref{Lag2} and \eqref{Lag3}, the system in Lagrangian coordinates is reduced as follows:
	\begin{equation}\label{Lag00}
		\begin{cases}
			&\frac{\partial^2 x^i}{\partial t^2}-b_0^k b_0^j \frac{\partial^2 x^i}{\partial y^k y^j}+ \partial_{x_i} q=0,
			\\
			&\mathrm{det}(\frac{\partial \bx}{\partial \by})=1,
			\\
			& \bx|_{t=0}=\by, \quad \frac{\partial \bx}{\partial t}|_{t=0}=\bv_0.
		\end{cases}
	\end{equation}
	Moreover, if we set
	\begin{equation*}
		H^{ia}=\frac{\partial x^i}{\partial y^a},
	\end{equation*}
	we also have
	\begin{equation}\label{LagB}
		\begin{cases}
			&\frac{\partial^2 H^{ia}}{\partial t^2}-b_0^k b_0^j \frac{\partial^2 H^{ia}}{\partial y^k y^j}+ \frac{\partial}{\partial y^a} (\partial_{x_i} q)=0,
			\\
			& H^{ia}|_{t=0}=\delta^{ia}, \quad \frac{\partial H^{ia}}{\partial t}|_{t=0}=\frac{\partial v^i_0}{\partial y^a},
		\end{cases}
	\end{equation}
	where $q$ and $\bH=(H^{ia})_{n\times n}$ also satisfy\footnote{Here $\bH^{-1}$ is a inverse matrix of $\bH$.}
	\begin{align}\label{LagC}
		\Delta_x {q}(t,\by)
		=& (\bH^{-1})^{{ki}} (\bH^{-1})^{mj} \big(  \frac{\partial H^{jk} }{\partial t}  \cdot \frac{\partial H^{im}}{\partial t}  
		-b_0^l b_0^r \frac{\partial H^{jk}}{\partial y^l} \cdot \frac{\partial H^{im}}{\partial y^r} \big),
		\\\label{LagD}
		(\bH^{-1})^{ia} = & \frac{1}{(n-1)!} \epsilon_{i i_2 i_3 \cdots i_n }\epsilon^{a a_2 a_3 \cdots a_n } H^{i_2 a_2}  \cdots H^{i_n a_n}.
	\end{align}
\end{Lemma}
\begin{proof}
	The incompressibility condition $\text{div}\bar{\bv}=0$ and \eqref{Lag2} mean that
	\begin{equation}\label{T00}
		\text{det}(\frac{\partial \bx}{\partial \by})=1.
	\end{equation}
	Using the chain rule and \eqref{Lag2}, we have
	\begin{equation*}\label{T01}
		\frac{d}{dt} ( \frac{\partial x^i(t,\by)}{\partial y^j} )= \frac{\partial \bar{v}^i}{\partial x^k}(t,\bx(t,\by)) \frac{\partial x^k(t,\by)}{\partial y^j}.
	\end{equation*}
	Similarly, we get
	\begin{equation}\label{T02}
		\frac{d}{dt} ( b^j_0(\by)\frac{\partial x^i(t,\by)}{\partial y^j} )=( b^j_0(\by)\frac{\partial x^k(t,\by)}{\partial y^j} ) \frac{\partial \bar{v}^i}{\partial x^k}(t,\bx(t,\by)) .
	\end{equation}
	By the second equation in \eqref{MHD}, it yields
	\begin{equation}\label{T03}
		\frac{d}{dt} ( \bar{b}^i(t,\bx(t,\by)) ) = \bar{b}^k(t,\bx(t,\by))  \frac{\partial \bar{v}^i}{\partial x^k}(t,\bx(t,\by)) {.}
	\end{equation}
	Due to $\bx(0,\by)=\by$, and \eqref{T02}-\eqref{T03}, we therefore get
	\begin{equation}\label{T030}
		\begin{cases}
			& \frac{d}{dt} ( \bar{b}^i - b^j_0\frac{\partial x^i}{\partial y^j} ) = ( \bar{b}^k - b^j_0\frac{\partial x^k}{\partial y^j} )  \frac{\partial \bar{v}^i}{\partial x^k} ,
			\\
			& (\bar{b}^i-b^j_0\frac{\partial x^i}{\partial y^j}) |_{t=0}= 0. 
		\end{cases}
	\end{equation}
	Since\footnote{When we derive the formulations for MHD equations in Lagrangian coordinates, we consider the smooth solutions, therefore \eqref{BKM} holds.} 
	\begin{equation}\label{BKM}
		\int^t_0 \| \nabla \bar{\bv} \|_{L^\infty(\mathbb{R}^n)} d\tau <\infty,
	\end{equation}
and $\bb_0$ is a constant vector, from \eqref{T030}, we can infer 
	\begin{equation}\label{T04}
		b^i(t,\by)=\bar{b}^i(t,\bx(t,\by)) = b^j_0 \frac{\partial x^{{i}}(t,\by)}{\partial y^j} .
	\end{equation}
	Furthermore, by the chain rule, we find
	\begin{equation}\label{T05}
		(\bar{b}^j \partial_j \bar{b}^i)(t,\bx(t,\by)) ) =\bar{b}^j(t,\bx(t,\by)) \frac{\partial \bar{b}^i}{ \partial y^k }(t,\bx(t,\by)) \cdot \frac{\partial y^k(t,\by)}{\partial x^j}.
	\end{equation}
	Inserting \eqref{T04} to \eqref{T05}, we have
	\begin{equation}\label{T06}
		\begin{split}
			(\bar{b}^j \partial_j \bar{b}^i)(t,\bx(t,\by)) ) =& b^l_0 \frac{\partial x^j}{\partial y^l} \frac{\partial \bar{b}^i}{ \partial y^k }(t,\bx(t,\by)) \cdot \frac{\partial y^k(t,\by)}{\partial x^j}
			\\
			= &b^l_0 \delta_{lk}\frac{\partial \bar{b}^i}{ \partial y^k }(t,\bx(t,\by)) 
			\\
			= &b^k_0  b^m_0 \frac{\partial^2 x^i}{ \partial y^k  \partial y^m}.
		\end{split}
	\end{equation}
	Therefore, in Lagrangian coordinates, since \eqref{T06}, we can write the first equation in \eqref{MHD} as
	\begin{equation}\label{T07}
		\frac{\partial^2 x^i}{\partial t^2}-b_0^k b_0^m \frac{\partial^2 x^i}{\partial y^k y^m}+ \partial_{x_i} q=0.
	\end{equation}
	Combining \eqref{Lag2}, \eqref{T00} and \eqref{T07}, we have proved \eqref{Lag00}. To prove \eqref{LagB}, let us calculate \eqref{Lag00} in a further way. By \eqref{MHD}, we infer
	\begin{equation*}
		\Delta_x \bar{q}=\partial_i \bar{v}^j \partial_j \bar{v}^i-\partial_i \bar{b}^j \partial_j \bar{b}^i.
	\end{equation*}
By equations \eqref{Lag2} and \eqref{Lag3}, and since $\mathbf{b}_0$ is a constant vector, we therefore obtain 
	\begin{equation}\label{T08}
		\begin{split}
			\Delta_x {q}(t,\by)=& \frac{\partial \bar{v}^j}{\partial y^k} \frac{\partial y^k}{\partial x^i} \cdot \frac{\partial \bar{v}^i}{\partial y^m} \frac{\partial y^m}{\partial x^j}  -\frac{\partial \bar{b}^j}{\partial y^k} \frac{\partial y^k}{\partial x^i} \cdot \frac{\partial \bar{b}^i}{\partial y^m} \frac{\partial y^m}{\partial x^j} 
			\\
			=& \frac{\partial y^k}{\partial x^i} \frac{\partial y^m}{\partial x^j} \left( \frac{\partial \bar{v}^j}{\partial y^k}  \frac{\partial \bar{v}^i}{\partial y^m}  -\frac{\partial \bar{b}^j}{\partial y^k} \frac{\partial \bar{b}^i}{\partial y^m} \right)
			\\
			=& \frac{\partial y^k}{\partial x^i} \frac{\partial y^m}{\partial x^j} \left\{  \frac{\partial }{\partial y^k} (\frac{\partial x^j}{\partial t} )  \frac{\partial }{\partial y^m} (\frac{\partial x^i}{\partial t} ) 
			-\frac{\partial }{\partial y^k}(b_0^l \frac{\partial x^j}{\partial y^l} ) \frac{\partial }{\partial y^m}(b_0^r \frac{\partial x^{{i}}}{\partial y^r} ) \right\}
			\\
			=& \frac{\partial y^k}{\partial x^i} \frac{\partial y^m}{\partial x^j} \left\{  \frac{\partial }{\partial t} (\frac{\partial x^j}{\partial y^k} ) \cdot \frac{\partial }{\partial t} (\frac{\partial x^i}{\partial y^m} ) 
			-b_0^l b_0^r \frac{\partial }{\partial y^l}( \frac{\partial x^j}{\partial y^k} ) \cdot \frac{\partial }{\partial y^r}( \frac{\partial x^{{i}}}{\partial y^m} ) \right\} .
		\end{split}
	\end{equation}
	For $\bH=(H^{ia})_{n\times n}$ being a $n\times n$ matrix, and $H^{ia}=\frac{\partial x^i}{ \partial y^a}$, using \eqref{T00}, then there exists an inverse matrix of $\bH$. We record it $\bH^{-1}$.  By \eqref{T08}, we deduce
	\begin{equation}\label{T10}
		\begin{split}
			\Delta_x {q}(t,\by)
			=& (\bH^{-1})^{ki} (\bH^{-1})^{mj} \left\{  \frac{\partial H^{jk} }{\partial t}  \cdot \frac{\partial H^{im}}{\partial t}  
			-b_0^l b_0^r \frac{\partial H^{jk}}{\partial y^l} \cdot \frac{\partial H^{im}}{\partial y^r} \right\}.
		\end{split}
	\end{equation}
	Operating $\frac{\partial }{ \partial y^a}$ on \eqref{T07}, it yields
	\begin{equation}\label{T11}
		\frac{\partial^2 H^{ia}}{\partial t^2}-b_0^k b_0^m \frac{\partial^2 H^{ia}}{\partial y^k y^m}+ \frac{\partial }{ \partial y^a}(\frac{\partial q}{ \partial x^i})=0.
	\end{equation}
	Using \eqref{T00} again, we have
	\begin{equation}\label{T12}
		\begin{split}
			(\bH^{-1})^{ia} = & \frac{1}{(n-1)!} \epsilon_{i i_2 i_3 \cdots i_n }\epsilon^{a a_2 a_3 \cdots a_n } H^{i_2 a_2}  \cdots H^{i_n a_n}.
		\end{split}
	\end{equation}
Combining \eqref{T10}, \eqref{T11}, and \eqref{T12}, we get \eqref{LagB}, \eqref{LagC} and \eqref{LagD}. Therefore, we complete the proof of this lemma.
\end{proof}

\section{Preliminaries}\label{pr}
In this section, we will introduce several inequalities within the solution space.
\begin{Lemma}\label{QR1}
	Let $\theta>\frac12$ and $a,b \in \mathbb{R}$. For any function $f\in \mathcal{S}'(\mathbb{R}^{1+n})$, we have 
	\begin{equation*}
		\|f\|_{ {L^\infty_t}H^{a,b} } \lesssim \|f\|_{ H^{a,b}_\theta } .
	\end{equation*}
\end{Lemma}
\begin{proof}
	Set $u=  \Lambda^a \Lambda^b_1 f$. Thus, we only need to prove the following estimate
	\begin{equation}\label{qr0}
		\| u \|_{L^2(\mathbb{R}^n)} \lesssim \|  \Lambda^\theta_{-} u\|_{L^2(\mathbb{R}^{1+n})}.
	\end{equation}
	A direct calculation and H\"older's inequality tell us
	\begin{equation}\label{qr1}
		\begin{split}
			\| u \|^2_{L^2(\mathbb{R}^n)} =& \int_{\mathbb{R}^n} |\widehat{u}(t,\xi)|^2d\xi
			\\
			=& \frac{1}{2\pi}\int_{\mathbb{R}^n} |\int_{\mathbb{R}} \widetilde{u}(\tau,\xi)\mathrm{e}^{i\tau t} d\tau |^2d\xi
			\\
			\lesssim & \int_{\mathbb{R}^n} \big( \int_{\mathbb{R}}  |\left<  |\tau|-|\xi_1| \right>^\theta \widetilde{u}(\tau,\xi) |^2 d\tau  \big) \cdot \big( \int_{\mathbb{R}}  \left<  |\tau|-|\xi_1| \right>^{-2\theta}  d\tau  \big)  d\xi.
		\end{split}
	\end{equation}
	Due to $\theta>\frac12$, it follows
	\begin{equation}\label{qr2}
		\int_{\mathbb{R}}  \left<  |\tau|-|\xi_1| \right>^{-2\theta}  d\tau  \lesssim 1.
	\end{equation}
	Inserting \eqref{qr2} to \eqref{qr1}, we therefore get \eqref{qr0}. At this stage, we complete the proof of this lemma.
\end{proof}
\begin{Lemma}\label{dQR}
	Let $s>\frac{n+1}{2} (n\geq 1)$ and $\theta>\frac12$. Then the following estimates
	\begin{equation}\label{dr0}
		\begin{split}
			\| f \|_{L^\infty(\mathbb{R}^{1+n})} \lesssim \| f \|_{H_\theta^{s-1,1}(\mathbb{R}^{1+n})},
		\end{split}
	\end{equation}
	and
	\begin{equation}\label{dr1}
		\begin{split}
			\| fg \|_{H^{s-1,0}_0 (\mathbb{R}^{1+n}) } \lesssim \| f\|_{H^{s-1,0}_0 (\mathbb{R}^{1+n})} \| g \|_{ H^{s-1,1}_\theta (\mathbb{R}^{1+n})},
		\end{split}
	\end{equation}
	hold.
\end{Lemma}
\begin{proof}
	By applying the space-time Fourier transform and H\"older's inequality, we obtain
	\begin{equation}\label{w1}
		\begin{split}
			\| f \|_{L^\infty(\mathbb{R}^{1+n})}
			\lesssim & \int_{ \mathbb{R}^{n} } \int_{\mathbb{R}} |\widetilde{f}(\tau,\xi)\mathrm{e}^{it\tau} | d\tau d\xi
			\\
			\lesssim & \left( \int_{ \mathbb{R}^{n} } \int_{\mathbb{R}} \left< |\tau|-|\xi_1| \right>^{2\theta} \left< \xi_1 \right>^2 \left<\xi\right>^{2(s-1)} |\widetilde{f}(\tau,\xi) |^2 d\tau d\xi \right)^{\frac12} 
			\\
			& \quad \times \left( \int_{ \mathbb{R}^{n} } \int_{\mathbb{R}} \left< |\tau|-|\xi_1| \right>^{-2\theta} \left< \xi_1 \right>^{-2} \left<\xi\right>^{-2(s-1)}  d\tau d\xi \right)^{\frac12}
			\\
			\lesssim & \| f \|_{H_\theta^{s-1,1}(\mathbb{R}^{1+n})} \left( \int_{ \mathbb{R}^{n} } \int_{\mathbb{R}} \left< |\tau|-|\xi_1| \right>^{-2\theta} \left< \xi_1 \right>^{-2} \left<\xi\right>^{-2(s-1)}  d\tau d\xi \right)^{\frac12}.
		\end{split}
	\end{equation}
	For $\theta>\frac12$ and $s>\frac{n+1}{2}$, we have
	\begin{equation}\label{w2}
		\left( \int_{ \mathbb{R}^{n} } \int_{\mathbb{R}} \left< |\tau|-|\xi_1| \right>^{-2\theta} \left< \xi_1 \right>^{-2} \left<\xi\right>^{-2(s-1)}  d\tau d\xi \right)^{\frac12} \lesssim 1.
	\end{equation}
	By estimates \eqref{w1} and \eqref{w2}, we have proven \eqref{dr0}. Note
	\begin{equation}\label{dr2}
		\Lambda^{s-1}(fg) \approx \Lambda^{s-1}f \cdot g +  f  \cdot \Lambda^{s-1}g .
	\end{equation}
	On one hand, by applying H\"older's inequality and using the estimate \eqref{dr0}, we can derive that
	\begin{equation}\label{dr3}
		\begin{split}
			\| \Lambda^{s-1}f \cdot g \|_{L^2(\mathbb{R}^{1+n})} \lesssim & \| \Lambda^{s-1}f \|_{L^2(\mathbb{R}^{1+n})} \| g \|_{L^\infty (\mathbb{R}^{1+n})}
			\\
			\lesssim & \| f \|_{H_0^{s-1,0}(\mathbb{R}^{1+n})} \| g \|_{H_\theta^{s-1,1}(\mathbb{R}^{1+n})}.
		\end{split}
	\end{equation}
	On the other hand, by H\"older's inequality, Lemma \ref{QR1}, and the Sobolev embedding theorem, we have
	\begin{equation}\label{dr4}
		\begin{split}
			\| \Lambda^{s-1}g \cdot f \|_{L^2(\mathbb{R}^{1+n})} \lesssim & \| \Lambda^{s-1}g \|_{L_{t,y_1}^\infty(\mathbb{R}^2)L^2(\mathbb{R}^{n-1})} \| f \|_{L^2_{t,y_1} (\mathbb{R}^2)L^\infty (\mathbb{R}^{n-1})}
			\\
			\lesssim & \|\Lambda_1 \Lambda^{s-1}g \|_{L_{t}^\infty(\mathbb{R})L^2(\mathbb{R}^{n})} \|\Lambda^{s-1} f \|_{L^2_{\mathbb{R}^{1+n}} }
			\\
			\lesssim & \| g \|_{H_\theta^{s-1,1}(\mathbb{R}^{1+n})} \| f \|_{H_0^{s-1,0}(\mathbb{R}^{1+n})}.
		\end{split}
	\end{equation}
	Combining \eqref{dr2}, \eqref{dr3}, and \eqref{dr4}, we have verified \eqref{dr1}. Thus, we complete the proof of this lemma.
\end{proof}
Next, we present a lemma concerning the inequalities between the Lagrangian and Eulerian coordinate systems.
\begin{Lemma}\label{LD4}
	Let $n\geq 1$. Denote $(t,\bx)$ by the Eulerian coordinates and $(t, \mathbf{y})$ as the Lagrangian coordinates. Let $u$ be a function from $(t,\by) \rightarrow \mathbb{R}^n$. Let $\bar{u}(t,\bx)=u(t,\by(t,\bx))$. If the Jacobian determinant satisfies $\text{det}(\frac{\partial \bx}{\partial \by})=1$, then the following properties hold: 
		\begin{equation}\label{ac1}
		\| u \|_{L^2(\mathbb{R}^n_y)} = \|\bar{u} \|_{L^2(\mathbb{R}^n_x)}, 
	\end{equation}
	\begin{equation}\label{ac2}
		\| u \|_{H^s(\mathbb{R}^n_y)} \leq C \| \frac{\partial \bx}{\partial \by} \|^{\frac{n}{2}+s}_{L^\infty(\mathbb{R}^n_x)} \|\bar{u} \|_{H^s(\mathbb{R}^n_x)}, \quad 0<s<1,
	\end{equation}
	and
	\begin{equation}\label{ac3}
		\| \bar{u}\|_{H^s(\mathbb{R}^n_x)} \leq C \| \frac{\partial \by}{\partial \bx} \|^{\frac{n}{2}+s}_{L^\infty(\mathbb{R}^n_y)} \| u \|_{H^s(\mathbb{R}^n_y)}, \quad 0<s<1,
	\end{equation}
	and
		\begin{equation}\label{ac4}
			\begin{split}
					&  \| \bar{u} \|_{H^1(\mathbb{R}^n_x)} \leq  (1+\| \frac{\partial \by}{\partial \bx} \|_{L^\infty(\mathbb{R}^n_x)})\| {u} \|_{H^1(\mathbb{R}^n_y)},
					\\
					&  \| {u} \|_{H^1(\mathbb{R}^n_y)} \leq  (1+\| \frac{\partial \bx}{\partial \by} \|_{L^\infty(\mathbb{R}^n_y)})\| \bar{u} \|_{H^1(\mathbb{R}^n_x)}.
			\end{split}
	\end{equation}
\end{Lemma}
\begin{proof}
	Firstly, by a change of coordinates, we have
	\begin{equation}\label{ab0}
		\begin{split}
			\| u(t,\cdot) \|_{L^2(\mathbb{R}^n)} &= \int_{\mathbb{R}^n_y} | u(t,\by) |^2 d\by
			\\
			& = \int_{\mathbb{R}^n} | \bar{u}(t,\bx) |^2  \text{det}(\frac{\partial \by}{\partial \bx}) d\bx
			\\
			& = \| \bar{u} \|_{L^2(\mathbb{R}^n_x)}.
		\end{split}
	\end{equation}
	For the homogeneous norm $\dot{H}^s$ ($0<s<1$), cf. \cite{BCD}, we have
	\begin{equation*}
		\begin{split}
			\| u(t,\cdot)  \|^2_{\dot{H}^s (\mathbb{R}^n_y)} &= \int_{\mathbb{R}^n} \int_{\mathbb{R}^n} \frac{|u(t,\by+\bh)-u(t,\by)|^2}{|\bh|^{n+2s}} d\by d\bh.
		\end{split}
	\end{equation*}
	Then we obtain
	\begin{equation}\label{ab1}
		\begin{split}
			\| u \|^2_{\dot{H}^s (\mathbb{R}^n_y)} &= \int_{\mathbb{R}^n} \int_{\mathbb{R}^n} \frac{|\bar{u}(t,x(t,\by+\bh))-\bar{u}(t,x(t,\by))|^2}{|x(t,\by+\bh)-x(t,\by)|^{n+2s}} \cdot \frac{|x(t,\by+\bh)-x(t,\by)|^{n+2s} }{|\bh|^{n+2s} } d\by d\bh
			\\
			&\leq \| \frac{\partial \bx}{\partial \by} \|^{n+2s}_{L^\infty(\mathbb{R}^n_y)} \cdot \int_{\mathbb{R}^n} \int_{\mathbb{R}^n} \frac{|\bar{u}(t,x(t,\by+\bh))-\bar{u}(t,x(t,\by))|^2}{|x(t,\by+\bh)-x(t,\by)|^{n+2s}} d\by d\bh .
		\end{split}
	\end{equation}
	On the other hand, a change of coordinates gives
	\begin{equation}\label{ab2}
		\begin{split}
			& \int_{\mathbb{R}^n} \int_{\mathbb{R}^n} \frac{|\bar{u}(t,x(t,\by+\bh))-\bar{u}(t,x(t,\by))|^2}{|x(t,\by+\bh)-x(t,\by)|^{n+2s}} d\by d\bh 
			\\
			=& \int_{\mathbb{R}^n} \int_{\mathbb{R}^n} \frac{|\bar{u}(t,x(t,\bar{\by}))-\bar{u}(t,x(t,\by))|^2}{|x(t,\bar{\by})-x(t,\by)|^{n+2s}} d\bar{\by} d\by 
			\\
			=& \int_{\mathbb{R}^n} \int_{\mathbb{R}^n} \left( \frac{|\bar{u}(t,x(t,\bar{\by}))-\bar{u}(t,\bx)|^2}{|x(t,\bar{\by})-\bx|^{n+2s}} \right)  \text{det}^{-1}(\frac{\partial \bx}{\partial \by}) d\bar{\by} d\bx 
			\\
			= & \int_{\mathbb{R}^n} \int_{\mathbb{R}^n} \left( \frac{|\bar{u}(t,\bar{\bx})-\bar{u}(t,\bx)|^2}{|\bar{\bx}-\bx|^{n+2s}} \right)  \text{det}^{-1}(\frac{\partial \bx(\bar{\by})}{\partial \bar{\by}}) d{\bar{\bx}} d\bx 
			\\
			= & \int_{\mathbb{R}^n} \int_{\mathbb{R}^n}  \frac{|\bar{u}(\bar{\bx})-\bar{u}(\bx)|^2}{|\bar{\bx}-\bx|^{n+2s}}  d{\bar{\bx}} d\bx .
		\end{split}
	\end{equation}
	where we set $x(t,\bar{\by})=\bar{\bx}$. Putting \eqref{ab0}, \eqref{ab1} and \eqref{ab2} together yields \eqref{ac1} and \eqref{ac2}. The inequalities \eqref{ac3} and \eqref{ac4} follow analogously.
\end{proof}
Finally, we present the following proposition.
\begin{proposition}\label{nE}
	Let $\theta \in (\frac12,1)$, $\varepsilon \in (0,1-\theta]$, and let $s\in \mathbb{R}$. The operator $\widetilde{\square}$ is defined in \eqref{wo}. Consider the Cauchy problem for the linear wave equation
	\begin{equation}\label{linearw}
		\begin{cases}
			& \widetilde{\square} U=F, \quad (t,y)\in \mathbb{R}^{1+n},
			\\
			& (u,\frac{\partial u}{\partial t})|_{t=0} =(f, g).
		\end{cases}
	\end{equation}
	Assume the functions $f,g, F$ satisfy 
	\begin{equation*}
		f\in H^{s-1,1}, \quad g\in H^{s-1,0},\quad F\in H^{s-1,0}_{\theta+\varepsilon-1}.
	\end{equation*}
	Let $\chi$ and $\phi$ be given by \eqref{chi} and \eqref{phi} respectively. For $0<T<1$ define
	\begin{equation}\label{defu}
		u=\chi(t) u_0+ \chi(\frac{t}{T}) u_1+ \chi(\frac{t}{T}) u_2,
	\end{equation}
	where
	\begin{equation}\label{defu0}
		\begin{split}
			u_0= &\cos(tD)f +  D^{-1} \sin(tD)g,
			\\
			F_1= & \phi( T^{\frac12}\Lambda_{-} )F, \quad F_2= ( 1-\phi( T^{\frac12}\Lambda_{-} ) )F,
			\\
			u_1= & \int^t_0  D^{-1} \sin( (t-t')D ) F_1(t')dt',
			\\
			u_2= & \widetilde{\square}^{-1} F_2 .
		\end{split}
	\end{equation}
Then, the function $u$ defined in \eqref{defu}-\eqref{defu0} is the unique solution of \eqref{linearw} on $[0,T]\times \mathbb{R}^n$ such that $u\in C([0,T];H^{s-1,1}) \cap C^1([0,T];H^{s-1,0})$. Moreover, the following estimate holds:
	\begin{equation}\label{none1}
		|u|_{H^{s-1,1}_\theta} \leq C(\|f\|_{H^{s-1,1}}+ \|g\|_{H^{s-1,0}}+ T^{\frac{\varepsilon}{2}}
		\| F\|_{H^{s-1,0}_{\theta+\varepsilon-1} } ),
	\end{equation}
	where the positive number $C$ only depends on $\chi$ and $\theta$.  
\end{proposition}
\begin{remark}
 We note that the function \( u \) is a solution of \eqref{linearw} on \([0,T] \times \mathbb{R}^n\), and that the $H^{s-1,1}_\theta$ norm of \( u \) is considered over the entire space-time domain \(\mathbb{R}^{1+n}\). The proof of Proposition \ref{nE} is presented in the appendix, Section \ref{App}. This proof is a minor modification of Theorem 12 from Selberg's paper \cite{Selberg}.
\end{remark} 

\section{Bilinear {estimate} of null forms}\label{be}
In this section, we establish a key bilinear estimate for the null form $Q_0(f, g)$, where $Q_0$ is defined in equation \eqref{nuf}.
\begin{Lemma}\label{eQR}
Let $s>\frac{n+1}{2}, \theta\in (\frac12,1)$ and $n \geq 2$. Suppose $f$ and $g$ belong to the space $H_{\theta}^{s-1,1}$. Then the following estimate holds:
	\begin{equation}\label{er2}
		\begin{split}
			\| Q_0(f,g)\|_{{H_0^{s-1,0}}} 
			\lesssim  & | f |_{H_{\theta}^{s-1,1}}  | g|_{H_{\theta}^{s-1,1}} .
		\end{split}
	\end{equation} 
\end{Lemma}
\begin{proof}
	By applying the Fourier transform on the space-time $\mathbb{R}^{1+n}$, we obtain
	\begin{equation}\label{kk2}
		\begin{split}
			\widetilde{Q_0(f,g)} (\tau,\xi)= \int_{\mathbb{R}^{n}} \int_{\mathbb{R}}  q_0 (\tau-\lambda,\xi-\eta,\lambda,\eta) 
			\widetilde{f} (\tau-\lambda,\xi-\eta)\widetilde{g}(\lambda,\eta) d\lambda d\eta,
		\end{split}
	\end{equation}
	where 
	\begin{equation*}
		q_0 (\tau-\lambda,\xi-\eta,\lambda,\eta) =(\tau- \lambda) \lambda - (\xi_1 - \eta_1) \eta_1.
	\end{equation*}
	From identity 
	\begin{equation*}
		(\tau-\lambda)\lambda=\frac12 \{ \tau^2-(\tau-\lambda)^2-\lambda^2\},
	\end{equation*}
	it follows that
	\begin{equation*}
		(\tau- \lambda) \lambda - (\xi_1 - \eta_1) \eta_1=\frac12 \left\{  \tau^2- \xi_1^2- (\tau-\lambda)^2 + (\xi_1-\eta_1)^2 -\lambda^2 + \eta_1^2  \right\}.
	\end{equation*}
Then, we deduce that
	\begin{equation}\label{kk1}
	\begin{split}
		 | q_0 (\tau-\lambda,\xi-\eta,\lambda,\eta)  | 
	\lesssim 	&  \left< |\tau|-|\xi_1| \right>  \left< |\tau|+|\xi_1| \right>  + \left< |\tau-\lambda|+|\xi_1-\eta_1| \right> \left< |\tau-\lambda|-|\xi_1-\eta_1| \right>
		\\
		& \ 
		+ \left< |\lambda|+|\eta_1| \right> \left< |\lambda|-|\eta_1| \right> .
	\end{split}	
	\end{equation}
On the other hand, we also have the following inequality:
	\begin{equation}\label{kk3}
	\begin{split}
		| q_0 (\tau-\lambda,\xi-\eta,\lambda,\eta)  | 
		\lesssim 	&  \left< |\tau-\lambda|+|\xi_1-\eta_1| \right> \left< |\lambda|+|\eta_1| \right>  .
	\end{split}	
\end{equation}
Applying the interpolation formula to \eqref{kk1} and \eqref{kk3} yields
\begin{equation}\label{kk5}
\begin{split}
	& | q_0 (\tau-\lambda,\xi-\eta,\lambda,\eta)  | 
		\\
		\lesssim 	&  \left< |\tau-\lambda|+|\xi_1-\eta_1| \right>^{\varepsilon}  \left< |\lambda|+|\eta_1| \right>^{\varepsilon} 
\cdot \big[   \left< |\tau|+|\xi_1| \right>^{1-\varepsilon} \left< |\tau|-|\xi_1| \right>^{1-\varepsilon}  
\\
& + \left< |\tau-\lambda|+|\xi_1-\eta_1| \right>^{1-\varepsilon} \left< |\tau-\lambda|-|\xi_1-\eta_1| \right>^{1-\varepsilon} + \left< |\lambda|+|\eta_1| \right>^{1-\varepsilon} \left< |\lambda|-|\eta_1| \right>^{1-\varepsilon}  
\big] 
\\
=& \left< |\tau|+|\xi_1| \right>^{1-\varepsilon} \left< |\tau|-|\xi_1| \right>^{1-\varepsilon} \left< |\tau-\lambda|+|\xi_1-\eta_1| \right>^{\varepsilon}  \left< |\lambda|+|\eta_1| \right>^{\varepsilon} 
\\
& + \left< |\tau-\lambda|+|\xi_1-\eta_1| \right> \left< |\tau-\lambda|-|\xi_1-\eta_1| \right>^{1-\varepsilon}  \left< |\lambda|+|\eta_1| \right>^{\varepsilon}   
\\
& + \left< |\tau-\lambda|+|\xi_1-\eta_1| \right>^{\varepsilon}  \left< |\lambda|+|\eta_1| \right> \left< |\lambda|-|\eta_1| \right>^{1-\varepsilon} ,
\end{split}	
\end{equation}
where we take $0<\varepsilon<\frac12$. By using \eqref{kk2} and \eqref{kk5}, we obtain
\begin{equation*}
	| Q_0(f,g) | \lesssim | \Lambda_{+}^{1-\varepsilon} \Lambda_{-}^{1-\varepsilon} \big(  \Lambda_{+}^{\varepsilon} f \cdot \Lambda_{+}^{\varepsilon} g \big) | + | \Lambda_{+} \Lambda_{-}^{1-\varepsilon} f \cdot \Lambda_{+}^{\varepsilon} g |
	+ | \Lambda_{+}^{\varepsilon} f \cdot \Lambda_{+} \Lambda_{-}^{1-\varepsilon} g | .
\end{equation*}
Note
\begin{equation*}
\begin{split}
\Lambda_{+}^{1-\varepsilon} \Lambda_{-}^{1-\varepsilon} \big(  \Lambda_{+}^{\varepsilon} f \cdot \Lambda_{+}^{\varepsilon} g \big) \approx
& \Lambda_{+} \Lambda_{-}^{1-\varepsilon} f \cdot \Lambda_{+}^{\varepsilon} g
+ \Lambda_{+}^{\varepsilon} f \cdot \Lambda_{+} \Lambda_{-}^{1-\varepsilon} g
\\
& + \Lambda_{-}^{1-\varepsilon} \Lambda_{+}^{\varepsilon} f \cdot \Lambda_{+}g + \Lambda_{+}f \cdot \Lambda_{-}^{1-\varepsilon} \Lambda_{+}^{\varepsilon}  g.
\end{split}	 
\end{equation*}
This implies
\begin{equation}\label{kk7}
	\begin{split}
		\| Q_0(f,g) \|_{L^2(\mathbb{R}^{1+n})} \lesssim & \| \Lambda_{-}^{1-\varepsilon} \Lambda_{+}^{\varepsilon} f \cdot \Lambda_{+}g \|_{L^2(\mathbb{R}^{1+n})}  + \|\Lambda_{+}f \cdot \Lambda_{-}^{1-\varepsilon} \Lambda_{+}^{\varepsilon}  g \|_{L^2(\mathbb{R}^{1+n})} 
		\\
		& + \| \Lambda_{+} \Lambda_{-}^{1-\varepsilon} f \cdot \Lambda_{+}^{\varepsilon} g \|_{L^2(\mathbb{R}^{1+n})}
		+ \| \Lambda_{+}^{\varepsilon} f \cdot \Lambda_{+} \Lambda_{-}^{1-\varepsilon} g \|_{L^2(\mathbb{R}^{1+n})}.
	\end{split}
\end{equation}
Take $\theta = 1 - \varepsilon$. Then we have $\theta \in \left(\frac{1}{2}, 1\right)$. Note that $s > \frac{n+1}{2}$. By H\"older's inequality and Sobolev embeddings, it follows that
\begin{equation}\label{kk9}
	\begin{split}
		\| \Lambda_{-}^{1-\varepsilon} \Lambda_{+}^{\varepsilon} f \cdot \Lambda_{+}g \|_{L^2(\mathbb{R}^{1+n})}  \lesssim & 
		\| \Lambda_{-}^{1-\varepsilon} \Lambda_{+}^{\varepsilon} f \|_{L^\infty_{t}L^2(\mathbb{R}^{n})} \|\Lambda_{+}g \|_{L^2_{t}L^{\infty}(\mathbb{R}^{n})}
		\\
		\lesssim & 
		\| \Lambda_{+}^{1-\varepsilon} \left( \Lambda_{-}^{1-\varepsilon} \Lambda_{+}^{\varepsilon} f \right) \|_{L^2_{t}L^2(\mathbb{R}^{n})} \|\Lambda_{-}^{1-\varepsilon} \Lambda^{s-1} \left( \Lambda_{+}g  \right) \|_{L^2_{t}L^2(\mathbb{R}^{n})}
		\\
		= & \| \Lambda_{+} \Lambda_{-}^{1-\varepsilon}  f \|_{L^2(\mathbb{R}^{1+n})} \|\Lambda_{-}^{1-\varepsilon} \Lambda^{s-1}  \Lambda_{+}g   \|_{L^2(\mathbb{R}^{1+n})}
		\\
		\lesssim & |f|_{H^{0,1}_{\theta}}  |g|_{H^{s-1,1}_{\theta}} .
	\end{split}
\end{equation}
Similarly, using H\"older's inequality and Sobolev embeddings again, we obtain
\begin{equation}\label{kk11}
	\begin{split}
	\|\Lambda_{+}f \cdot \Lambda_{-}^{1-\varepsilon} \Lambda_{+}^{\varepsilon}  g \|_{L^2(\mathbb{R}^{1+n})} 
	\lesssim & \|\Lambda_{+}f \|_{L^\infty_{t}L^2(\mathbb{R}^{n})} \| \Lambda_{-}^{1-\varepsilon} \Lambda_{+}^{\varepsilon}  g \|_{L^2_{t}L^{\infty}(\mathbb{R}^{n})} 
	\\
	\lesssim & 
	\| \Lambda_{-}^{1-\varepsilon} \Lambda_{+}f  \|_{L^2_{t}L^2(\mathbb{R}^{n})} \| \Lambda_{+}^{1-\varepsilon} \Lambda^{s-1}\left( \Lambda_{-}^{1-\varepsilon} \Lambda_{+}^{\varepsilon}  g \right) \|_{L^2_{t}L^{2}(\mathbb{R}^{n})}
	\\
	= & 
	\| \Lambda_{-}^{1-\varepsilon} \Lambda_{+}f  \|_{L^2(\mathbb{R}^{1+n})} \| \Lambda_{+} \Lambda^{s-1}  \Lambda_{-}^{1-\varepsilon}   g \|_{L^{2}(\mathbb{R}^{1+n})}
	\\
	\lesssim &
	 |f|_{H^{0,1}_{\theta}}  |g|_{H^{s-1,1}_{\theta}} ,
	\end{split}
\end{equation}
and
\begin{equation}\label{kk13}
	\begin{split}
		\| \Lambda_{+} \Lambda_{-}^{1-\varepsilon} f \cdot \Lambda_{+}^{\varepsilon} g \|_{L^2(\mathbb{R}^{1+n})} \lesssim & 
		\| \Lambda_{+} \Lambda_{-}^{1-\varepsilon} f \|_{L^2(\mathbb{R}^{1+n})} \| \Lambda_{+}^{\varepsilon} g \|_{L^\infty(\mathbb{R}^{1+n})}
		\\
		\lesssim & \| \Lambda_{+} \Lambda_{-}^{1-\varepsilon} f \|_{L^2(\mathbb{R}^{1+n})} \| \Lambda^{s-1} \Lambda_{-}^{1-\varepsilon}\Lambda_{+}^{1-\varepsilon} \left( \Lambda_{+}^{\varepsilon} g \right) \|_{L^2(\mathbb{R}^{1+n})}
		\\
		\lesssim & |f|_{H^{0,1}_{\theta}}  |g|_{H^{s-1,1}_{\theta}}  .
	\end{split}
\end{equation}
In a similar way, we also obtain
\begin{equation}\label{kk15}
	\begin{split}
		\| \Lambda_{+}^{\varepsilon} f \cdot \Lambda_{+} \Lambda_{-}^{1-\varepsilon} g \|_{L^2(\mathbb{R}^{1+n})} 
		\lesssim & \| \Lambda_{+}^{\varepsilon} f \|_{L^\infty_{t,y_1}L^2(\mathbb{R}^{n-1})} \| \Lambda_{+} \Lambda_{-}^{1-\varepsilon} g \|_{L^2_{t,y_1}L^\infty(\mathbb{R}^{n-1})} 
		\\
		\lesssim & \|  \Lambda_{+}^{1-\varepsilon} \Lambda_{-}^{1-\varepsilon} \Lambda_{+}^{\varepsilon} f \|_{L^2(\mathbb{R}^{1+n})} \| \Lambda^{s-1} \left( \Lambda_{+} \Lambda_{-}^{1-\varepsilon} g \right) \|_{L^2(\mathbb{R}^{1+n})} 
		\\
		\lesssim  & |f|_{H^{0,1}_{\theta}}  |g|_{H^{s-1,1}_{\theta}}   .
	\end{split}
\end{equation}
Submitting \eqref{kk9}, \eqref{kk11}, \eqref{kk13}, and \eqref{kk15} to \eqref{kk7}, we have
\begin{equation*}
	\begin{split}
		\| Q_0(f,g) \|_{L^2(\mathbb{R}^{1+n})} \lesssim &  |f|_{H^{0,1}_{\theta}}  |g|_{H^{s-1,1}_{\theta}}  .
	\end{split}
\end{equation*}
This implies that
\begin{equation*}
	\begin{split}
		\| Q_0(f,g) \|_{H^{s-1,0}_0} \lesssim &  |f|_{H^{s-1,1}_{\theta}}  |g|_{H^{s-1,1}_{\theta}}  .
	\end{split}
\end{equation*}
Therefore, the estimate \eqref{er2} holds, and we have completed the proof of Lemma \ref{eQR}.
\end{proof}

\section{Proof of Theorem \ref{thm}}\label{PT}
In this section we prove Theorem~\ref{thm} by means of the contraction mapping principle. Since the unknown in \eqref{Lag5}--\eqref{Lag5B} is the matrix-valued function $\bH$, it is convenient to introduce $\bG=\bH-\bE$, where $\bE$ denotes the $n\times n$ identity matrix. For simplicity of notation, we assume $\bb_0=(1,0,\cdots,0)^{\mathrm{T}}=e_1$.  Indeed, for any nonzero constant vector $\bb_0$, an orthogonal change of spatial coordinates transforms $\bb_0$ into $|\bb_0|e_1$, and a suitable rescaling of time converts the operator $\partial_t^2-(\bb_0\cdot\nabla_y)^2$ into $\partial_t^2-\partial_{y_1}^2$. Hence the same argument applies in general, with the existence time depending on $|\bb_0|$.

From \eqref{Lag5}--\eqref{Lag5B} we then obtain the following system for $\bG$:
\begin{equation}\label{p0}
	\begin{cases}
		& \widetilde{\square} G^{ia}=\frac{\partial }{\partial y^a} ( \frac{\partial q}{\partial x^i} ) ,
		\\
		& \Delta_x {q}=(\bH^{-1})^{kl}(\bH^{-1})^{mj} Q_0({H}^{\textcolor{red}{jk}}, {H}^{\textcolor{red}{lm}}) ,
		\\
		& 	 (G^{ia}, \frac{\partial G^{ia}}{\partial t})|_{t=0}=(0, \frac{\partial v_0^{i}}{\partial y^a}).
	\end{cases}
\end{equation}
where $\bH^{-1}=(\bG+\bE)^{-1}$ satisfies \eqref{Lag5B}. Let $s$ be as stated in Theorem \ref{thm}, and let the two positive constants be $$\theta=\frac34, \ \varepsilon=\frac14.$$ Define a solution space
\begin{equation}\label{Xd}
	 X_{s,\theta}=\left\{ u\in \mathcal{S}'(\mathbb{R}^{1+n}): |u|_{H^{s-1,1}_\theta}\leq C_1 \|\bv_0\|_{H^{s}} \right\},
\end{equation}
where $C_1$ will be defined later. For $\bG \in X_{s,\theta}$, we define the map $\text{M}$ as follows:
\begin{equation}\label{t0}
	\begin{split}
		\text{M} G^{ia}= 	&\chi(t) \left\{   D^{-1} \sin(tD) (\frac{\partial v^i_0}{\partial y^a}) \right\} +  { \widetilde{\square}^{-1}} ( 1-\phi( T^{\frac12}\Lambda_{-} ) )  \frac{\partial }{\partial y^a}(\frac{\partial q}{\partial x_i})
		\\
		& + \chi(\frac{t}{T}) \int^t_0 D^{-1} \sin( (t-t')D ) \left\{  \phi( T^{\frac12}\Lambda_{-} ) \frac{\partial }{\partial y^a}(\frac{\partial q}{\partial x_i}) \right\} (t')dt',
	\end{split}	
\end{equation}
where $q$ satisfies (please refer \eqref{Lag5A})
\begin{equation}\label{t00a}
	\Delta_x q=(\bH^{-1})^{kl}(\bH^{-1})^{mj}  Q_0({H}^{\textcolor{red}{jk}}, {H}^{\textcolor{red}{lm}}).
\end{equation}
Substituting $\bH=\bG+\bE$ into \eqref{t00a} yields
\begin{equation}\label{t00}
	\Delta_x q=((\bG+\bE)^{-1})^{kl}((\bG+\bE)^{-1})^{mj}  Q_0({G}^{\textcolor{red}{jk}}, {G}^{\textcolor{red}{lm}}).
\end{equation}
In \eqref{t0}, we can also calculate that $  \frac{\partial }{\partial y^a}(\frac{\partial q}{\partial x_i})= \frac{\partial x^j}{\partial y^a} \frac{\partial^2 q}{\partial x_i \partial x_j} =(\delta^{ja}+ G^{ja}) (-\Delta_x)^{-1}\partial_{x_i} \partial_{x_j} \Delta_x q$. Due to \eqref{t00}, we have
\begin{equation}\label{ty}
	\frac{\partial }{\partial y^a}(\frac{\partial q}{\partial x_i})= (\delta^{ja}+ G^{ja}) (-\Delta_x)^{-1}\partial_{x_i} \partial_{x_j} \left\{ ((\bG+\bE)^{-1})^{kl}((\bG+\bE)^{-1})^{mj}  Q_0(G^{jk}, G^{lm}) \right\}.
\end{equation}
Combining \eqref{t0} and \eqref{ty}, the map $\text{M}$ depends on $\bG$ explicitly. By \eqref{t0}, we can calculate out
\begin{equation*}\label{t1}
	\begin{cases}
		& \widetilde{\square}  \text{M} G^{ia}=\frac{\partial }{\partial y^a} ( \frac{\partial q}{\partial x^i} ), \quad \quad [0,T]\times \mathbb{R}^n,
		\\
		& (\text{M} G^{ia}, \frac{\partial } {\partial t} \text{M} G^{ia})|_{t=0}=(0,\frac{\partial v^i_0}{\partial y^a}). 
	\end{cases}
\end{equation*}
Using Proposition \ref{nE}, we derive
\begin{equation}\label{t2}
	\begin{split}
		| \text{M} \bG |_{H^{s-1,1}_\theta} \leq & C( \| \frac{\partial \bv_0}{\partial y} \|_{H^{s-1}} + T^{\frac{\varepsilon}{2}}  \| \frac{\partial }{\partial y}(\frac{\partial q}{\partial x}) \|_{H^{s-1,0}_{\theta+\varepsilon-1}}).
	\end{split}
\end{equation}
Thanks to $\frac{\partial }{\partial y}(\frac{\partial q}{\partial x})= \frac{\partial }{\partial x}(\frac{\partial q}{\partial x}) \cdot \frac{\partial x}{\partial y} $, we get
\begin{equation}\label{t20}
	\frac{\partial }{\partial y}(\frac{\partial q}{\partial x})
	= \bH\cdot \frac{\partial^2 q}{\partial x^2}
	= (\bE+\bG)\frac{\partial^2 q}{\partial x^2}.
\end{equation}
Substituting \eqref{t20} to \eqref{t2}, we can obtain
\begin{equation*}
	\begin{split}
		| \text{M} \bG |_{H^{s-1,1}_{\theta}} \leq  & C \left\{ \|\bv_0\|_{H^{s}}+ T^{\frac{\varepsilon}{2}}  ( \| \frac{\partial^2 q}{\partial^2 x} \|_{H^{s-1,0}_{\theta+\varepsilon-1}}+ \| \bG \frac{\partial^2 q}{\partial^2 x} \|_{H^{s-1,0}_{\theta+\varepsilon-1}} ) \right\} .
	\end{split}
\end{equation*}
For $\theta+ \varepsilon -1 = 0$, applying Lemma \ref{dQR} yields
\begin{equation}\label{t3}
	\begin{split}
		| \text{M} \bG |_{H^{s-1,1}_{\theta}} 
		\leq  & C \left\{  \|\bv_0\|_{H^{s}} + T^{\frac{\varepsilon}{2}}  (1+ | \bG |_{H^{s-1,1}_{\theta}}) \| \frac{\partial^2 q}{\partial^2 x} \|_{H^{s-1,0}_{\theta+\varepsilon-1} } \right\}.
	\end{split}
\end{equation}
For $\theta+ \varepsilon -1 = 0$, we also have
\begin{equation}\label{T0}
	\| \frac{\partial^2 q}{\partial^2 x} \|_{s-1,\theta+ \varepsilon -1 } = \| \frac{\partial^2 q}{\partial^2 x} \|_{s-1,0}=\| \frac{\partial^2 q}{\partial^2 x} \|_{L^2_{t}H^{s-1}(\mathbb{R}^n_y)}.
\end{equation}
Considering $s\in (\frac{n+1}{2}, \frac{n}{2}+1]$, we divide the discussion into two cases.

\textit{Case 1}: $n=2$. In this case, $s-1\in(\frac{1}{2},1]$, using Lemma \ref{LD4}, we can see
\begin{equation}\label{t4g}
	\begin{split}
		\| \frac{\partial^2 q}{\partial^2 x} \|_{L^2_{t}H^{s-1}(\mathbb{R}^n_y) } 
		\leq &  C (1+ \| \frac{\partial \bx}{\partial \by} \|^{\frac{n}{2}+s}_{L^\infty({\mathbb{R} \times \mathbb{R}^n_y) } } ) \| \frac{\partial^2 \bar{q}}{\partial^2 x} \|_{L^2_tH^{s-1}(\mathbb{R}^n_x)}
		\\
		\leq & C (1+ \| \frac{\partial \bx}{\partial \by} \|^{\frac{n}{2}+s}_{L^\infty({\mathbb{R} \times \mathbb{R}^n_y) } } )
		\| \Delta_x \bar{q}\|_{L^2_t H^{s-1}(\mathbb{R}^n_x)}
		\\
		\leq &  C (1+ \| \frac{\partial \bx}{\partial \by} \|^{\frac{n}{2}+s}_{L^\infty({\mathbb{R} \times \mathbb{R}^n_y) } } )
		\| \Delta_x {q}\|_{L^2_t H^{s-1}(\mathbb{R}^n_y)}  \| \frac{\partial \by}{\partial \bx} \|^{\frac{n}{2}+s}_{L^\infty({ \mathbb{R}\times} \mathbb{R}^n_y)} 
		\\
		\leq & C (1+|\bG|^{\frac{n}{2}+s}_{H^{s-1,1}_\theta} ) (1+ |\bG|^{2(n-1)}_{H^{s-1,1}_\theta} ) \| Q_0(\bG,\bG) \|_{H^{s-1,0}_0} (1+|\bG|_{L^\infty({  \mathbb{R}\times } \mathbb{R}^n_y)}^{n-1})^{\frac{n}{2}+s} 
		\\
		\leq & C (1+|\bG|^{n^2+3n}_{H^{s-1,1}_\theta}  ) \| Q_0(\bG,\bG) \|_{H^{s-1,0}_0}.
	\end{split}
\end{equation}
\textit{Case 2}: $n=3,4$. In this case, $s-2\in(0,1]$, using Lemma \ref{LD4} again, we obtain
\begin{equation*}\label{t4a}
	\begin{split}
		\| \frac{\partial^2 q}{\partial^2 x} \|_{L^2_{t}H^{s-1}(\mathbb{R}^n_y) } =& \| \frac{\partial}{\partial y} (\frac{\partial^2 q}{\partial^2 x}) \cdot  (\bG+\bE) \|_{L^2_{t}H^{s-2}(\mathbb{R}^n_y) } + \| \frac{\partial^2 q}{\partial^2 x} \|_{L^2_{t}L^{2}(\mathbb{R}^n_y) } 
		\\
		\leq & \| \frac{\partial^3 q}{\partial^3 x}  \|_{L^2_{t}H^{s-2}(\mathbb{R}^n_y) } (1+|\bG|^2_{H^{s-1,1}_\theta}) + \| \frac{\partial^2 q}{\partial^2 x} \|_{L^2_{t}L^{2}(\mathbb{R}^n_y) }.
	\end{split}
\end{equation*}
For the term $\frac{\partial^3 q}{\partial^3 x}$, we can derive
\begin{equation}\label{t4b}
	\begin{split}
		\| \frac{\partial^3 q}{\partial^3 x}  \|_{L^2_{t}H^{s-2}(\mathbb{R}^n_y) } \leq C \| \frac{\partial \bx}{\partial \by} \|^{\frac{n}{2}+s-2}_{L^\infty({\mathbb{R} \times \mathbb{R}^n_y) } } \| \frac{\partial^3 \bar{q}}{\partial^3 x} \|_{L^2_tH^{s-2}(\mathbb{R}^n_x)},
	\end{split}
\end{equation}
and
\begin{equation}\label{t4d}
	\begin{split}
		\| \frac{\partial^3 \bar{q}}{\partial^3 x} \|_{L^2_tH^{s-2}(\mathbb{R}^n_x)}\leq & \| \frac{\partial }{\partial x} \Delta_x \bar{q} \|_{L^2_tH^{s-2}(\mathbb{R}^n_x)}
		\\
		= & \| \frac{\partial }{\partial y} \Delta_x \bar{q} \cdot (\bG+\bE)^{-1} \|_{L^2_tH^{s-2}(\mathbb{R}^n_x)}
		\\
		\leq 
		& C\| \frac{\partial \by}{\partial \bx} \|^{\frac{n}{2}+s-2}_{L^\infty({\mathbb{R} \times \mathbb{R}^n_x) }} \| \frac{\partial }{\partial y} \Delta_x \bar{q} \cdot (\bG+\bE)^{-1} \|_{L^2_tH^{s-2}(\mathbb{R}^n_y)} 
		\\
		\leq 
		& C\| \frac{\partial \by}{\partial \bx} \|^{\frac{n}{2}+s-2}_{L^\infty({\mathbb{R} \times \mathbb{R}^n_y) } }(1+|\bG|^{n-1}_{H^{s-1,1}_\theta}) \| \Delta_x {q} \|_{L^2_tH^{s-1}(\mathbb{R}^n_y)}.
	\end{split}
\end{equation}
Above, we use the embedding $H^{s-1,1}_{\theta}(\mathbb{R}^{1+n}) \hookrightarrow L^\infty(\mathbb{R}^{1+n})$, which has proved in Lemma \ref{dQR}. Inserting \eqref{t4d} into \eqref{t4b}, we obtain
\begin{equation}\label{t4m}
	\begin{split}
		\| \frac{\partial^3 {q}}{\partial^3 x} \|_{L^2_tH^{s-2}(\mathbb{R}^n_y)}\leq  C(1+|\bG|^{n^2+3n}_{H^{s-1,1}_\theta}) \| \Delta_x {q} \|_{L^2_tH^{s-1}(\mathbb{R}^n_y)}.
	\end{split}
\end{equation}
Moreover, we also have
\begin{equation}\label{t4c}
	\begin{split}
		\| \frac{\partial^2 q}{\partial^2 x} \|_{L^2_{t}L^{2}(\mathbb{R}^n_y) } 
		= 
		\| \Delta_x \bar{q}\|_{L^2_t L^{2}(\mathbb{R}^n_x)}
		=
		\| \Delta_x {q}\|_{L^2_t L^{2}(\mathbb{R}^n_y)}.
	\end{split}
\end{equation}
Combing \eqref{t4m} and \eqref{t4c}, we conclude that
\begin{equation}\label{t4f}
	\begin{split}
		\| \frac{\partial^2 q}{\partial^2 x} \|_{L^2_{t}H^{s-1}(\mathbb{R}^n_y) } 
		\leq & C  (1+ {|\bG|^{n^2+3n}_{H^{s-1,1}_\theta} } ) \| Q_0(\bG,\bG) \|_{H^{s-1,0}_0}.
	\end{split}
\end{equation}
By \eqref{t4g} and \eqref{t4f}, in both $n=2$ and $n=3,4$, we prove
\begin{equation}\label{t4}
	\begin{split}
		\| \frac{\partial^2 q}{\partial^2 x} \|_{L^2_{t}H^{s-1}(\mathbb{R}^n_y) } 
		\leq & C (1+  |\bG|^{n^2+3n}_{H^{s-1,1}_\theta}  ) \| Q_0(\bG,\bG) \|_{H^{s-1,0}_0}.
	\end{split}
\end{equation}
Due to \eqref{T0} and \eqref{t4}, it implies
\begin{equation}\label{t70}
	\| \frac{\partial^2 q}{\partial^2 x} \|_{s-1,\theta+ \varepsilon -1 }\leq  C  (1+  |\bG|^{n^2+3n}_{H^{s-1,1}_\theta}  ) \| Q_0(\bG,\bG) \|_{H^{s-1,0}_0}.
\end{equation}
Applying Lemma \ref{eQR} yields:
\begin{equation}\label{t5}
	\| Q_0(\bG,\bG) \|_{H^{s-1,0}_0} \leq C |\bG|^2_{H^{s-1,1}_\theta}. 
\end{equation}
Combining \eqref{t3}, \eqref{t70}, and \eqref{t5} yields 
\begin{equation}\label{t6}
	\begin{split}
		| \text{M} \bG |_{H^{s-1,1}_{\theta}} \leq   & C \|\bv_0\|_{H^{s}} 
		+ CT^{\frac{\varepsilon}{2}} (1+|\bG|^{n^2+3n+3}_{{H^{s-1,1}_\theta}}) .
	\end{split}
\end{equation}
{To verify that $\text{M}$ is a contraction mapping on $X_{s,\theta}$, we need to show that if $| \bG |_{H^{s-1,1}_{\theta}} \leq C_1 \|\bv_0\|_{H^{s}} $, then}
\begin{equation*}
	{ | \text{M} \bG |_{H^{s-1,1}_{\theta}} \leq C_1 \|\bv_0\|_{H^{s}}. }
\end{equation*}
Let the positive constant $C$ be given as in \eqref{t6}. Set $C_1=2C$ in \eqref{Xd} and choose 
\begin{equation}\label{CT}
	T=\min\left\{  \frac12, \left[  \frac{\|\bv_0\|_{H^{s}} }{ 1+( 2C \|\bv_0\|_{H^{s}} )^{n^2+3n+3}} \right]^{\frac{2}{\varepsilon}} \right\} .
\end{equation}
Then we can calculate the estimate \eqref{t6} by
\begin{equation}\label{t7}
	\begin{split}
		| \text{M} \bG |_{H^{s-1,1}_{\theta}} \leq   & C \|\bv_0\|_{H^{s}} 
	+ C\frac{\|\bv_0\|_{H^{s}} }{ 1+( 2C \|\bv_0\|_{H^{s}} )^{n^2+3n+3}} \left\{  1+( 2C \|\bv_0\|_{H^{s}} )^{n^2+3n+3}) \right\}
	\\
	 =   & 2C \|\bv_0\|_{H^{s}} 
	 =  C_1\|\bv_0\|_{H^{s}} .
	\end{split}
\end{equation}
By \eqref{t7}, $\text{M}$ is a map from $X_{s,\theta}$ to $X_{s,\theta}$. Next, we will prove that $\text{M}$ is a contraction map in $X_{s,\theta}$. For
\begin{equation*}
	Q_0(\bG,\bG)-Q_0(\bJ,\bJ)=Q_0(\bG-\bJ,\bG)+Q_0({\bJ},\bG-\bJ),
\end{equation*}
we therefore conclude
\begin{equation}\label{t8}
	\begin{split}
		| \text{M} \bG -  \text{M} \bJ |_{H^{s-1,1}_{\theta}} \leq   & C  T^{\frac{\varepsilon}{2}}  |\bG - \bJ|_{H^{s-1,1}_{\theta}} \left(1+|\bG|^{n^2+3n+2}_{H^{s-1,1}_{\theta}} +|\bJ|^{n^2+3n+2}_{H^{s-1,1}_{\theta}} \right).
	\end{split}
\end{equation}
When $T$ is sufficiently small, using \eqref{t8}, we shall get 
\begin{equation*}\label{t9}
	\begin{split}
		| \text{M} \bG -  \text{M} \bJ |_{H^{s-1,1}_{\theta}} \leq   & \frac12  |\bG - \bJ |_{H^{s-1,1}_{\theta}}.
	\end{split}
\end{equation*}
Therefore, $\text{M}$ is a contraction mapping in $X_{s,\theta}$. Using the contraction mapping principle, then there is a unique solution for the problem \eqref{p0}. To prove the continuous dependence, let $\bG_1$ and $\bG_2$ satisfy
\begin{equation*}\label{G1}
	\begin{cases}
		& \widetilde{\square} G_1^{ia}=\frac{\partial }{\partial y^a} ( \frac{\partial p_1}{\partial x^i} ) ,
		\\
		& \Delta_x p_1=(\bF_1^{-1})^{kl}(\bF_1^{-1})^{mj} Q_0({F}_1^{\textcolor{red}{jk}}, {F}_1^{\textcolor{red}{lm}}) ,
		\\
		& 	 (G_1^{ia}, \frac{\partial G_1^{ia}}{\partial t})|_{t=0}=(0, \frac{\partial v_{01}^{i}}{\partial y^a}),
	\end{cases}
\end{equation*}
and
\begin{equation*}\label{G2}
	\begin{cases}
		& \widetilde{\square} G_2^{ia}=\frac{\partial }{\partial y^a} ( \frac{\partial p_2}{\partial x^i} ) ,
		\\
		& \Delta_x p_2=(\bF_2^{-1})^{kl}(\bF_2^{-1})^{mj} Q_0({F}_2^{jk}, {F}_2^{lm}) ,
		\\
		& 	 (G_2^{ia}, \frac{\partial G_2^{ia}}{\partial t})|_{t=0}=(0, \frac{\partial v_{02}^{i}}{\partial y^a}),
	\end{cases}
\end{equation*}
where $\bF_1^{-1}$ and $\bF_2^{-1}$ satisfy
\begin{equation*}\label{G3}
	\begin{split}
		(\bF_1^{-1})^{ia} = & \frac{1}{(n-1)!} \epsilon_{i i_2 i_3 \cdots i_n }\epsilon^{a a_2 a_3 \cdots a_n } (\delta^{i_2 a_2}+ G_1^{i_2 a_2}) \cdots (\delta^{i_n a_n}+ G_1^{i_n a_n}),
		\\
		(\bF_2^{-1})^{ia} = & \frac{1}{(n-1)!} \epsilon_{i i_2 i_3 \cdots i_n }\epsilon^{a a_2 a_3 \cdots a_n } (\delta^{i_2 a_2}+ G_2^{i_2 a_2}) \cdots (\delta^{i_n a_n}+ G_2^{i_n a_n}).
	\end{split}
\end{equation*}
Using {Proposition \ref{nE}}, we can prove
\begin{equation}\label{G4}
	\begin{split}
		& |\bG_2-\bG_1|_{H^{s-1,1}_{\theta}} 
		\\
		\leq C & \|\bv_{02}-\bv_{01}\|_{H^{s}} 
		+CT^{\frac{\varepsilon}{2} } |\bG_2 - \bG_1|_{H^{s-1,1}_{\theta}} \left(1+ |\bG_1|^{n^2+3n+2}_{H^{s-1,1}_{\theta}}+ |\bG_2|^{n^2+3n+2}_{H^{s-1,1}_{\theta}}  \right).
	\end{split}
\end{equation}
By \eqref{G4}, for $\bG_1, \bG_2\in X_{s,\theta}$, if $T$ is sufficiently small, then we can get
\begin{equation}\label{G5}
	\begin{split}
		|\bG_2-\bG_1|_{H^{s-1,1}_{\theta}} \leq C   \|\bv_{02}-\bv_{01}\|_{H^{s}} .
	\end{split}
\end{equation}
Using Lemma \ref{QR1} and estimate \eqref{G5} yield
\begin{equation*}\label{G6}
	\begin{split}
		\|\bG_2-\bG_1\|_{L^\infty_{[0,T]}H^{s-1,1}}+\|\frac{\partial }{\partial t} (\bG_2-\bG_1 )\|_{L^\infty_{[0,T]}H^{s-1}} 
		\leq  C \|\bv_{02}-\bv_{01}\|_{H^{s}} .
	\end{split}
\end{equation*}
Therefore, we have proved the continuous dependence of the solutions on the initial data. This completes the proof of Theorem \ref{thm}.


\section{Appendix: The proof of Proposition \ref{nE}}\label{App}
In this section, we present a proof of Proposition \ref{nE}, inspired by Selberg's paper \cite{Selberg}. By examining equations \eqref{defu} and \eqref{none1}, we need to establish the corresponding estimates for $\chi(t)u_0$, $\chi(t)u_1$, and $\chi(t)u_2$, respectively.
\subsection{The estimates for $\chi(t)u_0$}
By \eqref{defu0}, we recall $\chi(t)u_0= \chi(t) \cos(tD)f +  \chi(t) D^{-1} \sin(tD)g$. Let us introduce the following lemma:
\begin{Lemma}\label{CQR}
	Let $\theta>\frac12$ and $s\in \mathbb{R}$. Let $\chi \in C^\infty_c(\mathbb{R})$ and $(f,g)\in H^{s-1,1}\times H^{s-1,0}$. Then the following estimates hold:
	\begin{align}
		\label{cr0}
		\| \chi(t)\mathrm{e}^{\pm \mathrm{i}tD}f \|_{H^{s-1,1}_{\theta}(\mathbb{R}^{1+n})} \lesssim & \| \chi\|_{H^\theta(\mathbb{R})}\|f \|_{H^{s-1,1}(\mathbb{R}^n)},
		\\ \label{cr1}
		\| \chi(t) \cos(tD) f \|_{H^{s-1,1}_{\theta}(\mathbb{R}^{1+n})} \lesssim & \| \chi\|_{H^\theta(\mathbb{R})}\|f \|_{H^{s-1,1}(\mathbb{R}^n)}.
	\end{align}
	If $r\in [-1,1]$, and $\mathrm{supp} \widehat{f} \subseteq \{\xi: |\xi_1|<m\}$ ($m>0$), then we also have
	\begin{equation}\label{cr5}
		\| \chi(t)\mathrm{e}^{ \mathrm{i}rtD} f \|_{H^{s-1,1}_{\theta}(\mathbb{R}^{1+n})} \lesssim (m^\theta\|\chi\|_{L^2(\mathbb{R})}+\|\chi\|_{H^\theta(\mathbb{R})} ) \|f\|_{H^{s-1,1}(\mathbb{R}^n)}.
	\end{equation}
	Moreover, we get 
	\begin{align}\label{cr2}
		\| \chi(t) D^{-1}\sin(tD) g \|_{H_{\theta}^{s-1,1} (\mathbb{R}^{1+n})} \lesssim & (\| \chi\|_{H^\theta(\mathbb{R})}+ \| t \chi\|_{H^\theta(\mathbb{R})}) \|g \|_{H^{s-1,0}(\mathbb{R}^n)}.
	\end{align}
\end{Lemma}
\begin{proof}
	Since the space-time Fourier transform of $\chi(t)\mathrm{e}^{\pm \mathrm{i}tD}f$ is $\widehat{\chi}(\tau \mp |\xi_1|)\widehat{f}(\xi)$, we obtain
	\begin{equation}\label{cr4}
		\begin{split}
			& \| \widehat{\chi}(\tau \mp |\xi_1|)\widehat{f}(\xi) \|_{H^{s-1,1}_{\theta}(\mathbb{R}^{1+n})}
			\\
			=&  \left( \int_{\mathbb{R}^{1+n}} \left<\xi\right>^{2(s-1)} \left<\xi_1 \right>^{2} \left<|\tau|-|\xi_1| \right>^{2\theta} |\widehat{\chi}(\tau\mp|\xi_1|) \widehat{f}(\xi)|^2 d\tau d\xi\right)^{\frac12} 
			\\
			=&  \left( \int_{\mathbb{R}^{n}}\left<\xi\right>^{2(s-1)}\left<\xi_1 \right>^{2}  |\widehat{f}(\xi)|^2 d\xi \int_{\mathbb{R}}  \left<|\tau|-|\xi_1| \right>^{2\theta} |\widehat{\chi}(\tau\mp|\xi_1|) |^2 d\tau  \right)^{\frac12} 
			\\
			\lesssim & \|\chi\|_{H^\theta} \|f\|_{H^{s-1,1}}.
		\end{split}
	\end{equation}
	Thus, the estimate \eqref{cr0} is established. By using the identity $\cos(tD) f=\frac12(\mathrm{e}^{ \mathrm{i}tD}f + \mathrm{e}^{- \mathrm{i}tD}f )$ and \eqref{cr4}, we can directly obtain \eqref{cr1}. Note that the space-time Fourier transform of $\chi(t)\mathrm{e}^{ \mathrm{i}rtD} {f} $ equals 
	\begin{equation*}
		\widehat{\chi}(\tau-r |\xi_1|)\widehat{{f}}(\xi) .
	\end{equation*}
	When $\xi \in \mathrm{supp} \widehat{f}$ and $r\in[-1,1]$, it yields
	\begin{equation*}\label{cr6}
		| |\tau|-|\xi_1| | \leq |  \tau- r|\xi_1| | + (1-|r|)|\xi_1| \leq |  \tau- r|\xi_1| |+m.
	\end{equation*}
	Through calculations, we can deduce that
	\begin{equation*}\label{cr7}
		\begin{split}
			 \| \chi(t)\mathrm{e}^{ \mathrm{i}rtD} f \|^2_{H^{s-1,1}_{\theta}(\mathbb{R}^{1+n})} 
			=& \int_{\mathbb{R}^{1+n}} \left< |\tau|-|\xi_1| \right>^{2\theta} | \widehat{\chi}(\tau-r|\xi_1|) |^2 \left< \xi_1 \right>^2 \left<\xi\right>^{2(s-1)} | \widehat{f}(\xi) |^2 d\tau d\xi
			\\
			\lesssim & \int_{\mathbb{R}^{1+n}} \left< |\tau|-|\xi_1| \right>^{2\theta} | \widehat{\chi}(\tau-r|\xi_1|) |^2  \left< \xi_1 \right>^2 \left<\xi\right>^{2(s-1)} | \widehat{f}(\xi) |^2 d\tau d\xi
			\\
			\lesssim & \int_{\mathbb{R}^{1+n}} (\left< \tau -r|\xi_1| \right>^{2\theta}+m^{2\theta} ) | \widehat{\chi}(\tau-r|\xi_1|) |^2  \left< \xi_1 \right>^2 \left<\xi\right>^{2(s-1)} | \widehat{{f}}(\xi) |^2 d\tau d\xi
			\\
			\lesssim & (\|\chi\|^2_{H^\theta}+ m^{2\theta}\|\chi\|^2_{L^2} ) \|f\|^2_{H^{s-1,1}}.
		\end{split}
	\end{equation*}
Thus, we have proved \eqref{cr5}. It remains for us to prove \eqref{cr2}. Give a decomposition
	\begin{equation*}
		g=g_1+g_2, \quad \textrm{supp} \ \widehat{g}_1 \subseteq \{ \xi: |\xi_1|< 1\} , \ \textrm{and} \  \textrm{supp} \ \widehat{g}_2 \subseteq \{ \xi: |\xi_1|\geq 1\}.
	\end{equation*}
We rewrite $\chi(t)D^{-1} \sin(tD) g_1$ as
	\begin{equation*}
		\chi(t)D^{-1} \sin(tD) g_1=\int^1_0 t\chi(t) \mathrm{e}^{ \mathrm{i}(2r-1)tD} g_1 dr, \quad  \sup \widehat{g}_1 \subseteq \{ \xi: |\xi_1|< 1\}.
	\end{equation*}
Due to \eqref{cr5}, we can see
	\begin{equation}\label{cr8}
		\| \chi(t)D^{-1} \sin(tD) g_1 \|_{H^{s-1,1}_\theta} \lesssim \| t \chi\|_{H^\theta}\| g_1 \|_{H^{s-1,0}}.
	\end{equation}
Noting that $\widehat{g}_2 \subseteq \{ \xi: |\xi_1|\geq 1\}$, we therefore obtain
	\begin{equation}\label{cr9}
		\begin{split}
			\| \chi(t)D^{-1} \sin(tD) g_2 \|_{H^{s-1,1}_\theta} 
			=& 
			\| \chi(t) \sin(tD) (D^{-1} g_2) \|_{H^{s-1,1}_\theta}
			\\
			\lesssim & \| \chi\|_{H^\theta}\| D^{-1} g_2 \|_{H^{s-1,1}}
			\\
			\lesssim & \| \chi\|_{H^\theta}\| g_2 \|_{H^{s-1,0}}.
		\end{split}
	\end{equation}
Combining \eqref{cr8} with \eqref{cr9} shows that \eqref{cr2} holds. Therefore, the proof of Lemma \ref{CQR} is complete.
\end{proof}
\subsection{The estimates for $u_2$}
By \eqref{defu0}, we have $ u_2 = \widetilde{\square}^{-1} F_2 $, where $F_2= (1-\phi(\Lambda_-) )F$. We define
\begin{equation}\label{nset}
	\mathcal{N}=\left\{ (\tau,\xi)\in \mathbb{R}^{1+n}: \left||\tau|-|\xi_1| \right|<1 \right\}.
\end{equation}
\begin{Lemma}\label{mF}
	Assume $s\in \mathbb{R}$, $\theta \in (\frac12,1)$. Let $u_2$ and $F_2$ be stated in \eqref{defu0}. Then the following estimate holds:
	\begin{equation}\label{up00}
		|u_2|_{H^{s-1,1}_\theta} \leq C
		\| F_2\|_{H^{s-1,0}_{\theta-1} }.
	\end{equation}
\end{Lemma}
\begin{proof}
Note that $\textrm{supp} \widetilde{F}_2 \subseteq \mathbb{R}^{1+n}\backslash \mathcal{N}$, where $\mathcal{N}$ is defined in \eqref{nset}. From \eqref{defu0}, we obtain
	\begin{align*}
		\widetilde{u}_2(\tau,\xi)=& (|\tau|^2-|\xi_1|^2)^{-1} \widetilde{F}_2(\tau,\xi),
		\\
		\widetilde{\partial_t u_2}(\tau,\xi)=& \tau (|\tau|^2-|\xi_1|^2)^{-1} \widetilde{F}_2(\tau,\xi).
	\end{align*}
	Consequently, we can compute
	\begin{equation}\label{up03}
		\begin{split}
			\| {u}_2\|_{H^{s-1,1}_{\theta}}\lesssim  & \| \left<\xi\right>^{s-1} \left< \xi_1 \right> \left< |\tau|-|\xi_1| \right>^\theta (|\tau|-|\xi_1|)^{-1}	(|\tau|+|\xi_1|)^{-1} \widetilde{F}_2(\tau,\xi) \|_{L^2(\mathbb{R}^{1+n})}
			\\
			\lesssim & \| \left<\xi\right>^{s-1} \left< |\tau|-|\xi_1| \right>^{\theta-1} \widetilde{F}_2(\tau,\xi) \|_{L^2(\mathbb{R}^{1+n})}=\| F_2\|_{H^{s-1,0}_{\theta-1} }.
		\end{split}
	\end{equation}
	Similarly, we have
	\begin{equation}\label{up05}
		\begin{split}
			\| \partial_t {u}_2\|_{H^{s-1,0}_{\theta}} \lesssim  & \| \left<\xi\right>^{s-1}  \left< |\tau|-|\xi_1| \right>^\theta |\tau| (|\tau|-|\xi_1|)^{-1}	(|\tau|+|\xi_1|)^{-1} \widetilde{F}_2(\tau,\xi) \|_{L^2(\mathbb{R}^{1+n})}
			\\
			\lesssim & \| \left<\xi\right>^{s-1} \left< |\tau|-|\xi_1| \right>^{\theta-1} \widetilde{F}_2(\tau,\xi) \|_{L^2(\mathbb{R}^{1+n})}=\| F_2\|_{H^{s-1,0}_{\theta-1} }.
		\end{split}
	\end{equation}
	Due to estimates \eqref{up03} and \eqref{up05}, we obtain \eqref{up00}. Therefore, the proof of this lemma is complete.
\end{proof}
\subsection{The estimates for $\chi(t)u_1$}
By \eqref{defu0}, we recall
\begin{equation*}\label{lwa}
	u_1(t)={ \int^t_0 D^{-1}\sin( (t-t')D ) \cdot F_1(t') dt' }.
\end{equation*} 
where $F_1$ is defined as in \eqref{defu0}. To bound $u_1$, let us first introduce the following decomposition.
\begin{Lemma}\label{mE}
	Let $s\in \mathbb{R}$, $\theta \in (\frac12,1)$ and $c_0$ be a positive constant and $c_0\geq 2$. Suppose that
	\begin{equation}\label{lwh}
		2+||\tau|-|\xi_1|| \leq c_0, \ \textrm{for} \ (\tau,\xi)\in \mathrm{supp} \ \widetilde{F}_1.
	\end{equation}
	Then there exists $f_j^{\pm} \in H^{s-1,1}$, $g_j \in C([0,1],H^{s-1,0})$ for $j\geq 1$ such that
	\begin{equation}\label{U1a}
		\begin{split}
			& \mathrm{supp} \ \widehat{f_j^{\pm}} \subseteq \{\xi: |\xi_1| \geq c_0\}, \qquad \qquad \qquad \mathrm{supp} \ \widehat{g_j} \subseteq \{\xi: |\xi_1| < c_0\},
			\\
			& \| f_j^{\pm}\|_{H^{s-1,1}}\lesssim (3c_0)^{j-\frac12} \|F_1\|_{H_0^{s-1,0}}, \quad \sup_{\rho\in[0,1]}\|g_j(\rho)\|_{H^{s-1,0}}  \lesssim (3c_0)^{j-\frac12} \|F_1\|_{H_0^{s-1,0}},
		\end{split}
	\end{equation}
	and
	\begin{equation}\label{U1}
		\begin{split}
			u_1(t)=&\sum^{\infty}_{j=1} \frac{t^{j+1}}{j!} \int^1_0 \mathrm{e}^{\mathrm{i}t(2\rho-1)D} g_j(\rho)d\rho +\sum^{\infty}_{j=1} \frac{t^{j}}{j!} ( \mathrm{e}^{\mathrm{i}tD}f^+_j + \mathrm{e}^{-\mathrm{i}tD}f^-_j )+R_+(t)+R_-(t).
		\end{split}
	\end{equation}
	Above, $R_+(t)$ and $R_-(t)$ are given by
	\begin{equation}\label{Rf}
		\begin{split}
			{{\widehat{R}}}_{+}(t)=& -\frac{1}{4\pi |\xi_1|} \int^0_{-\infty} \frac{\mathrm{e}^{\mathrm{it\tau}}-\mathrm{e}^{\mathrm{it|\xi_1|}} }{|\tau|+|\xi_1|} \widetilde{F_{1,2}}(\tau,\xi)d\tau,
			\\
			{{\widehat{R}}}_-(t)=& -\frac{1}{4\pi |\xi_1|} \int_0^{\infty} \frac{\mathrm{e}^{\mathrm{it\tau}}-\mathrm{e}^{-\mathrm{it|\xi_1|}} }{|\tau|+|\xi_1|} \widetilde{F_{1,2}}(\tau,\xi)d\tau,
		\end{split}
	\end{equation}
	where $\textrm{supp} \ \widehat{F_{1,2}} \subseteq \{ \xi: |\xi_1|\geq c_0 \}$. Moreover, there exists $h^{\pm}_j \in H^{s-1,0}$ for $j\geq 1$ such that
	\begin{equation}\label{hj}
		\begin{split}
			\| h^{\pm}_j \|_{ H^{s-1,0}} \lesssim (3c_0)^{j-\frac12} \| F_1\|_{H_0^{s-1,0}},
		\end{split}
	\end{equation}
	and
	\begin{equation}\label{hj1}
		\begin{split}
			\partial_t	u_1(t)=&\sum^{\infty}_{j=1} \frac{t^{j}}{j!} ( \mathrm{e}^{\mathrm{i}tD}h^+_j + \mathrm{e}^{-\mathrm{i}tD}h^-_j )-\mathrm{i}DQ_+(t)+\mathrm{i}D Q_-(t),
		\end{split}
	\end{equation}
	and
	\begin{equation*}\label{hj2}
		\begin{split}
			{{\widehat{Q}}}_{+}(t)=  &-\frac{1}{4\pi |\xi_1|} \int^0_{-\infty} \frac{\mathrm{e}^{\mathrm{it\tau}}-\mathrm{e}^{\mathrm{it|\xi_1|}} }{|\tau|+|\xi_1|} \widetilde{F_{1}}(\tau,\xi)d\tau,
			\\
			{{\widehat{Q}}}_{-}(t)=  & -\frac{1}{4\pi |\xi_1|} \int_0^{\infty} \frac{\mathrm{e}^{\mathrm{it\tau}}-\mathrm{e}^{-\mathrm{it|\xi_1|}} }{|\tau|+|\xi_1|} \widetilde{F_{1}}(\tau,\xi)d\tau.
		\end{split}
	\end{equation*}
\end{Lemma}
\begin{proof}
We define
	\begin{equation}\label{b01}
		G_{\pm}(t)= \int^t_0 \mathrm{e}^{\pm \mathrm{i} (t-t')D} \cdot F_1(t') dt'.
	\end{equation}
Then, we can calculate the Fourier transform of $G$ by
	\begin{equation}\label{b02}
		\begin{split}
			\widehat{G}_{\pm}(t,\xi)= & \int^t_0 \mathrm{e}^{\pm \mathrm{i} (t-t')|\xi_1|} \cdot \widehat{F}_1(t') dt'.
			\\
			=& \int^t_0 \mathrm{e}^{\pm \mathrm{i} (t-t')|\xi_1|} \left(\frac{1}{2\pi} \int_{\mathbb{R}} \mathrm{e}^{\mathrm{i}t'\tau}\widetilde{F}_1(\tau,\xi)d\tau \right) dt'
			\\
			=& \frac{ \mathrm{e}^{\pm \mathrm{i} t|\xi_1|} }{2\pi}  \int_{\mathbb{R}} \left( \int^t_0 \mathrm{e}^{ \mathrm{i} t'(\tau \mp |\xi_1|)}dt' \right)  \widetilde{F}_1(\tau,\xi) d\tau
			\\
			=& \frac{ \mathrm{e}^{\pm \mathrm{i} t|\xi_1|} }{2\pi}  \int_{\mathbb{R}} \frac{\mathrm{e}^{ \mathrm{i} t(\tau \mp |\xi_1|)}-1}{\mathrm{i}(\tau \mp |\xi_1|)}   \widetilde{F}_1(\tau,\xi) d\tau .
		\end{split}	
	\end{equation}	
	By Taylor's expansion, we have:
	\begin{equation}\label{b03}
		\begin{split}
			\frac{\mathrm{e}^{ \mathrm{i} t(\tau \mp |\xi_1|)}-1}{\mathrm{i}(\tau \mp |\xi_1|)}= & \sum^{\infty}_{j=1} \frac{t^j}{j!}\mathrm{i}^{j-1} (\tau \mp |\xi_1|)^{j-1}.
		\end{split}	
	\end{equation}	
	Substituting \eqref{b03} into \eqref{b02}, we obtain
	\begin{equation}\label{b05}
		\begin{split}
			\widehat{G}_{\pm}(t,\xi)
			=& \frac{ \mathrm{e}^{\pm \mathrm{i} t|\xi_1|} }{2\pi}  \sum^{\infty}_{j=1} \frac{t^j}{j!}  \int_{\mathbb{R}} \mathrm{i}^{j-1} (\tau \mp |\xi_1|)^{j-1}  \widetilde{F}_1(\tau,\xi) d\tau .
		\end{split}	
	\end{equation}	
	Since there is an operator $D^{-1}$ in $u_1$, so we need to consider low and high frequencies separately. To do this, we decompose $F_1=F_{1,1}+F_{1,2}$, where 
	\begin{equation}\label{B04}
		\mathrm{supp}\widehat{F}_{1,1}\subseteq \{\xi: |\xi_1| < c_0 \}, \quad \mathrm{supp}\widehat{F}_{1,2}\subseteq \{\xi : |\xi_1| \geq c_0 \}.
	\end{equation}
	Let $u_{1,a}$ be defined as 
	\begin{equation*}
		u_{1,a}(t)={ \int^t_0 D^{-1}\sin( (t-t')D ) \cdot F_{1,a}(t') dt'}.
	\end{equation*}
	Then, we can rewrite $u_{1,a}$ as
	\begin{equation}\label{B05}
		u_{1,a}(t)={ \frac{1}{2\mathrm{i}}\int^t_0 D^{-1} \left\{  \mathrm{e}^{\mathrm{i}(t-t')D} F_{1,a}(t') - \mathrm{e}^{-\mathrm{i}(t-t')D} F_{1,a}(t') \right\} dt' }, \quad a=1,2.
	\end{equation}
Next, we divide the problem into several steps to bound \( u_{1,a} \) for \( a = 1, 2 \) and \(\partial_t u_1\).

	\textbf{Step 1: formulation for $u_{1,1}$}. Define
	\begin{equation*}
		\alpha(r)=\mathrm{e}^{\mathrm{i}tr}(\tau-r)^{j-1}.
	\end{equation*}
	Observing equations \eqref{b01} and \eqref{b05}, and using \eqref{B05}, we obtain
	\begin{equation}\label{b06}
		\begin{split}
			\widehat{u}_{1,1}(t,\xi)=& \textcolor{red}{-}\frac{1}{4\pi} \sum^{\infty}_{j=1} \frac{t^j}{j!} 
			\int_{\mathbb{R} } \mathrm{i}^j |\xi_1|^{-1} \left( \mathrm{e}^{\mathrm{i}t|\xi_1|} (\tau-|\xi_1|)^{j-1}
			- \mathrm{e}^{-\mathrm{i}t|\xi_1|} (\tau+|\xi_1|)^{j-1} \right) \widetilde{F_{1,1}} (\tau,\xi) d\tau
			\\
			=& \textcolor{red}{-}\frac{1}{4\pi} \sum^{\infty}_{j=1} \frac{t^j}{j!} \int_{\mathbb{R} } \mathrm{i}^j |\xi_1|^{-1} 
			\left(\alpha(|\xi_1|)-\alpha(-|\xi_1|)\right) \widetilde{F_{1,1}}(\tau,\xi) d\tau
			\\
			=& \textcolor{red}{-}\frac{1}{2\pi} \sum^{\infty}_{j=1} \frac{t^j}{j!} \int_{\mathbb{R} } \int^1_0 \mathrm{i}^j  
			\alpha'((2\rho-1)|\xi_1|)  \widetilde{F_{1,1}} (\tau,\xi)d\rho d\tau
			\\
			=& \textcolor{red}{-}\frac{1}{2\pi} \sum^{\infty}_{j=1} \frac{t^j}{j!} \int^1_0 \int_{\mathbb{R} }  \mathrm{i}^j 
			\alpha'((2\rho-1)|\xi_1|)  \widetilde{F_{1,1}} (\tau,\xi) d\tau d\rho.
		\end{split}
	\end{equation} 
	By direct calculation, we have
	\begin{equation}\label{b07}
		\alpha'(r)=\mathrm{i}t\mathrm{e}^{\mathrm{i}tr}(\tau-r)^{j-1}-(j-1)\mathrm{e}^{\mathrm{i}tr}(\tau-r)^{j-2}.
	\end{equation}
	Above, the second term exists only for $ j \geq 2$. Denote $k_j(\rho)$ such that
	\begin{equation}\label{b08}
		\widehat{k}_j(\rho,\xi)= \textcolor{red}{-}\int_{\mathbb{R} }  (\tau-(2\rho-1)|\xi_1|)^{j-1} \widetilde{F_{1,1}} (\tau,\xi) d\tau.
	\end{equation}
	Since equations \eqref{B04} and \eqref{b08} hold, it is easy to see that
	\begin{equation}\label{b09}
		\mathrm{supp} \ \widehat{k}_j(\rho,\xi) \subseteq \{\xi : |\xi_1| < c_0\}.
	\end{equation}
	Inserting \eqref{b07} and \eqref{b08} into \eqref{b06} yields
	\begin{equation}\label{b10}
		\begin{split}
			{u}_{1,1}(t)=&  \frac{1}{2\pi} \sum^{\infty}_{j=1} \frac{t^j}{j!} \int^1_0  \mathrm{i}^{j+1} t \mathrm{e}^{\mathrm{i}t(2\rho-1)D} k_j (\rho) d\rho
			\\
			& - \frac{1}{2\pi} \sum^{\infty}_{j=1} \frac{t^j}{j!} \int^1_0  \mathrm{i}^{j} (j-1)  \mathrm{e}^{\mathrm{i}t(2\rho-1)D} k_{j-1} (\rho) d\rho
			\\
			=& \frac{1}{2\pi} \sum^{\infty}_{j=1} \frac{t^{j+1}}{j!} \int^1_0  \mathrm{i}^{j+1}  (1-\frac{j}{j+1})\mathrm{e}^{\mathrm{i}t(2\rho-1)D} k_j (\rho) d\rho \textcolor{red}{.}
		\end{split}
	\end{equation}
	Define
	\begin{equation}\label{b11}
		g_j(\rho)= (2\pi)^{-1}\mathrm{i}^{j+1}  (1-\frac{j}{j+1}) k_j(\rho).
	\end{equation}
	By \eqref{b09}, we have $\mathrm{supp} \ \widehat{g}_j(\rho,\xi) \subseteq \{\xi: |\xi_1| < c_0\}$.
	Using \eqref{b10} and \eqref{b11}, we can rewrite $u_{1,1}$ as
	\begin{equation}\label{b12}
		u_{1,1}(t)= \sum^{\infty}_{j=1} \frac{t^{j+1}}{j!} \int^1_0  \mathrm{e}^{\mathrm{i}t(2\rho-1)D} g_j (\rho) d\rho .
	\end{equation}
	For $\rho\in[0,1]$ and $\xi \in \mathrm{supp}\ \widehat{g}_j(\rho)$, we have
	\begin{equation*}\label{b13}
		|\tau-(2\rho-1)|\xi_1||\leq |\tau|+|\xi_1| \leq | |\tau|-|\xi_1| |+ 2|\xi_1| \leq | |\tau|-|\xi_1| |+2c_0 \leq 3c_0.
	\end{equation*}
	Combining \eqref{b08}, \eqref{b11}, \eqref{b13}, and applying H\"older's inequality, we obtain
	\begin{equation}\label{B13a}
		\begin{split}
			\|g_j(\rho)\|_{H^{s-1,0}} \lesssim & \left( \int_{|\tau|\leq 3c_0 }  |\tau|^{2(j-1)} d\tau \right)^{\frac12} 
			\left( \int_{\mathbb{R} }  \| \widetilde{F_{1,1}} (\tau,\xi)\|^2_{H^{s-1,0}} d\tau \right)^{\frac12} 
			\\
			\lesssim & (3c_0)^{j-\frac12} \| F_1 \|_{H^{s-1,0}_0}. 
		\end{split}
	\end{equation} 
	\textbf{Step 2: formulation for $u_{1,2}$}. By using equations \eqref{b01} and \eqref{b03}, we can obtain
	\begin{equation}\label{B16}
		\begin{split}
			\widehat{G}_{+}(t,\xi)=& 
			\frac{\mathrm{e}^{\mathrm{it|\xi_1|}}}{2\pi} \int^\infty_0  \frac{\mathrm{e}^{\mathrm{it(\tau-|\xi_1|)}}{-1}}{\mathrm{i}(\tau-|\xi_1|)} \widetilde{F}_{1}(\tau,\xi)d\tau
			+ \frac{\mathrm{e}^{\mathrm{it|\xi_1|}}}{2\pi} \int_{-\infty}^0  \frac{\mathrm{e}^{\mathrm{it(\tau-|\xi_1|)}}{-1}}{\mathrm{i}(\tau-|\xi_1|)} \widetilde{F}_{1}(\tau,\xi)d\tau
			\\
			=& 	\frac{\mathrm{e^{\mathrm{i}t|\xi_1|}}}{2\pi } \sum^\infty_{j=1} \frac{t^j}{j!} 
			\int^\infty_0 \mathrm{i}^{j-1} (|\tau|-|\xi_1|)^{j-1} \widetilde{F_{1}}(\tau,\xi)d\tau
			- \frac{1}{2\pi} \int^0_{-\infty} \frac{\mathrm{e}^{\mathrm{it\tau}}-\mathrm{e}^{\mathrm{it|\xi_1|}}}{\mathrm{i}(|\tau|+|\xi_1|)}\widetilde{F_{1}}(\tau,\xi)d\tau.
		\end{split}
	\end{equation}
	In a similar way, we can also obtain
	\begin{equation}\label{B17}
		\begin{split}
			\widehat{G}_{-}(t,\xi)
			=& 	\frac{\mathrm{e^{-\mathrm{i}t|\xi_1|}}}{2\pi } \sum^\infty_{j=1} \frac{t^j}{j!} 
			\int_{-\infty}^0 \mathrm{i}^{j-1} (\tau+|\xi_1|)^{j-1} \widetilde{F_{1}}(\tau,\xi)d\tau
			+ \frac{1}{2\pi} \int_0^{\infty} \frac{\mathrm{e}^{\mathrm{it\tau}}-\mathrm{e}^{-\mathrm{it|\xi_1|}}}{\mathrm{i}(\tau+|\xi_1|)}\widetilde{F_{1}}(\tau,\xi)d\tau
			\\
			=& 	\frac{\mathrm{e^{-\mathrm{i}t|\xi_1|}}}{2\pi } \sum^\infty_{j=1} \frac{t^j}{j!} 
			\int_{-\infty}^0 \mathrm{i}^{j-1} (|\xi_1|-|\tau|)^{j-1} \widetilde{F_{1}}(\tau,\xi)d\tau
			+ \frac{1}{2\pi} \int_0^{\infty} \frac{\mathrm{e}^{\mathrm{it\tau}}-\mathrm{e}^{-\mathrm{it|\xi_1|}}}{\mathrm{i}(|\tau|+|\xi_1|)}\widetilde{F_{1}}(\tau,\xi)d\tau.
		\end{split}
	\end{equation}
	Due to equations \eqref{B05}, \eqref{b01}, \eqref{B16}, and \eqref{B17}, we have
	\begin{equation*}
		\begin{split}
			\widehat{u}_{1,2}(t,\xi)=& 
			\frac{\mathrm{e^{\mathrm{i}t|\xi_1|}}}{4\pi |\xi_1|} \sum^\infty_{j=1} \frac{t^j}{j!} 
			\int^\infty_0 \mathrm{i}^{j} (|\tau|-|\xi_1|)^{j-1} \widetilde{F_{1,2}}(\tau,\xi)d\tau - \frac{1}{4\pi |\xi_1|} \int^0_{-\infty} \frac{\mathrm{e}^{\mathrm{it\tau}}-\mathrm{e}^{\mathrm{it|\xi_1|}}}{|\tau|+|\xi_1|} \widetilde{F_{1,2}}(\tau,\xi)d\tau
			\\
			& -\frac{\mathrm{e^{-\mathrm{i}t|\xi_1|}}}{4\pi |\xi_1|} \sum^\infty_{j=1} \frac{t^j}{j!} 
			\int_{-\infty}^0 \mathrm{i}^{j} (|\xi_1|-|\tau|)^{j-1} \widetilde{F_{1,2}}(\tau,\xi)d\tau
			- \frac{1}{4\pi |\xi_1|} \int_0^{\infty} \frac{\mathrm{e}^{\mathrm{it\tau}}-\mathrm{e}^{-\mathrm{it|\xi_1|}}}{|\tau|+|\xi_1|} \widetilde{F_{1,2}}(\tau,\xi)d\tau
		\end{split}
	\end{equation*}
	Therefore, we obtain
	\begin{equation}\label{U12}
		\begin{split}
			u_{1,2}(t)=&\sum^{\infty}_{j=1} \frac{t^{j}}{j!} ( \mathrm{e}^{\mathrm{i}tD}f^+_j + \mathrm{e}^{-\mathrm{i}tD}f^-_j )+R_+(t)+R_-(t),
		\end{split}
	\end{equation}
	where
	\begin{equation*}
		\begin{split}
			\widehat{f_j^+}(\xi)=& 
			( 4\pi |\xi_1| )^{-1}   
			\int^\infty_0 \mathrm{i}^{j} (|\tau|-|\xi_1|)^{j-1} \widetilde{F_{1,2}}(\tau,\xi)d\tau,
			\\
			\widehat{f_j^-}(\xi)= & -( 4\pi |\xi_1| )^{-1}  
			\int_{-\infty}^0 \mathrm{i}^{j} (|\xi_1|-|\tau|)^{j-1} \widetilde{F_{1,2}}(\tau,\xi)d\tau,
			\\
			\widehat{R_{+}}(t,\xi)=& - \frac{1}{4\pi |\xi_1|} \int^0_{-\infty} \frac{\mathrm{e}^{\mathrm{it\tau}}-\mathrm{e}^{\mathrm{it|\xi_1|}}}{|\tau|+|\xi_1|} \widetilde{F_{1,2}}(\tau,\xi)d\tau,
			\\
			\widehat{R_{-}}(t,\xi)=&- \frac{1}{4\pi |\xi_1|} \int_0^{\infty} \frac{\mathrm{e}^{\mathrm{it\tau}}-\mathrm{e}^{-\mathrm{it|\xi_1|}}}{|\tau|+|\xi_1|} \widetilde{F_{1,2}}(\tau,\xi)d\tau.
		\end{split}
	\end{equation*}
	Using \eqref{lwh}, \eqref{B04}, and H\"older's inequality, for $j\geq 1$, we can derive
	\begin{equation}\label{B14}
		\begin{split}
			\|{f_j^{\pm}}\|_{H^{s-1,1}} \lesssim c_0^{j-\frac12} \|{F_{1,2}}\|_{H^{s-1,0}_0} \lesssim (3c_0)^{j-\frac12} \|{F_{1,2}}\|_{H^{s-1,0}_0}.
		\end{split}
	\end{equation}
	By \eqref{b12}, \eqref{U12}, \eqref{B13a}, and \eqref{B14}, we obtain \eqref{U1} and \eqref{U1a}.

	\textbf{Step 3: formulation for $\partial_t u_1$}. We begin by calculating
	\begin{equation*}
		\begin{split}
			\partial_t u_1=& \int^t_0 \cos\left((t-t')D\right)F_1(t')dt'
			\\
			=& \frac12 \int^t_0 \left( \mathrm{e}^{\mathrm{i}(t-t')D}F_1(t') + \mathrm{e}^{-\mathrm{i}(t-t')D}F_1(t') \right) dt'.
		\end{split}
	\end{equation*}
	By \eqref{b02}, we have
	\begin{equation*}
		\begin{split}
			\widehat{\partial_t u_1}(t,\xi)
			=& \frac{\mathrm{e}^{\mathrm{i}t|\xi_1|}}{4\pi} \int_{\mathbb{R}} \frac{\mathrm{e}^{\mathrm{i}t(\tau-|\xi_1|)}-1}{\mathrm{i}(\tau-|\xi_1|)} \widetilde{F}_1(\tau,\xi)d\tau
			+\frac{\mathrm{e}^{-\mathrm{i}t|\xi_1|}}{4\pi} \int_{\mathbb{R}} \frac{\mathrm{e}^{\mathrm{i}t(\tau+|\xi_1|)}-1}{\mathrm{i}(\tau+|\xi_1|)} \widetilde{F}_1(\tau,\xi)d\tau.
		\end{split}
	\end{equation*}
	By \eqref{b05}, we get
	\begin{equation*}
		\begin{split}
			\widehat{\partial_t u_1}(t,\xi)
			=&  \frac{1}{4\pi} \int_{0}^{\infty} \mathrm{i}^{j}(\tau-|\xi_1|)^{j-1} \widetilde{F}_1(\tau,\xi)d\tau 
			-\frac{1}{4\pi}
			\int_{-\infty}^0 \mathrm{i}^{j} (|\xi_1|-|\tau|)^{j-1} \widetilde{F}_1 (\tau,\xi)d\tau,
			\\
			& -\frac{\mathrm{i}}{4\pi } \int_{0}^{\infty} \frac{\mathrm{e}^{\mathrm{i}t\tau}-\mathrm{e}^{-\mathrm{i}t|\xi_1|}}{|\tau|+|\xi_1|} \widetilde{F}_1(\tau,\xi)d\tau
			-\frac{\mathrm{i}}{4\pi } \int_{-\infty}^{0} \frac{\mathrm{e}^{\mathrm{i}t\tau}-\mathrm{e}^{\mathrm{i}t|\xi_1|}}{|\tau|+|\xi_1|} \widetilde{F}_1(\tau,\xi)d\tau.
		\end{split}
	\end{equation*}
	Hence, we can conclude that
	\begin{equation*}
		\begin{split}
			\partial_t	u_1(t)=&\sum^{\infty}_{j=1} \frac{t^{j}}{j!} ( \mathrm{e}^{\mathrm{i}tD}h^+_j + \mathrm{e}^{-\mathrm{i}tD}h^-_j )-\mathrm{i}DQ_+(t)+\mathrm{i}D Q_-(t),
		\end{split}
	\end{equation*}
	where $Q_\pm$ is defined in Lemma \ref{hj2}, and
	\begin{equation*}
		\begin{split}
			\widehat{h_j^+}(\xi)=& 
			\frac{\mathrm{i}}{4\pi }  
			\int^\infty_0 \mathrm{i}^{j} (|\tau|-|\xi_1|)^{j-1} \widetilde{F_{1}}(\tau,\xi)d\tau,
			\\
			\widehat{h_j^-}(\xi)= & -\frac{\mathrm{i}}{4\pi }  
			\int_{-\infty}^0 \mathrm{i}^{j} (|\xi_1|-|\tau|)^{j-1} \widetilde{F_{1}}(\tau,\xi)d\tau .
		\end{split}
	\end{equation*}
	By \eqref{lwh} and H\"older's inequality, we can deduce that
	\begin{equation*}
		\begin{split}
			\|{h_j^{\pm}}\|_{H^{s-1,0}} \lesssim &  (3c_0)^{j-\frac12} \|F_{1} \|_{H^{s-1,0}_0}.
		\end{split}
	\end{equation*}
	At this stage, we have completed the proof of Lemma \ref{mE}.
\end{proof}
Next, we provide an energy estimate for $\chi(t) u_1$.
\begin{Lemma}\label{oE}
	Assume $s\in \mathbb{R}$ and $\theta \in (\frac12,1)$. Let $u_1$ and $F_1$ be defined as in \eqref{defu0}, and suppose \eqref{lwh} holds. Then, the following estimate holds:
	\begin{equation}\label{op0}
		|\chi u_1|_{H^{s-1,1}_\theta} \leq C_0
		\| F_1\|_{H^{s-1,0}_{0} },
	\end{equation}
	where $C_0$ is given by
	\begin{equation*}
		\begin{split}
			C_0\approx \ & (3c_0)^{\frac12} \left( \|\chi\|_{\dot{H}^{\theta-1}}+\|t\chi\|_{H^{\theta}}+\|t\chi'\|_{\dot{H}^{\theta-1}}+\|t^2\chi'\|_{H^{\theta}}  \right)
			\\
			& +\sum^\infty_{j=1} \big( \frac{(3c_0)^{j+\frac12}}{j!} \| t^{j+1}\chi \|_{H^\theta}+ \frac{(3c_0)^{j+\frac12+\theta}}{j!} \| t^{j+1}\chi \|_{L^2} 
			\\
			& \quad +\frac{(3c_0)^{j-\frac12}}{j!} \| t^{j+1}\chi' \|_{H^\theta}
			+ \frac{(3c_0)^{j-\frac12+\theta}}{j!} \| t^{j+1}\chi \|_{L^2} 
			+ \frac{(3c_0)^{j-\frac12}}{j!} \| t^{j}\chi \|_{H^\theta} \big).
		\end{split}
	\end{equation*}
\end{Lemma}
\begin{proof}
	To estimate $|\chi u_1|_{H^{s-1,1}_\theta}$, we need to consider $\|\chi u_1\|_{H^{s-1,1}_\theta}, \|\chi' u_1\|_{H^{s-1,0}_\theta}$ and $\|\chi \partial_t u_1\|_{H^{s-1,0}_\theta}$. 
	We divide the proof into several steps.

	\textbf{Step 1: Estimate for $\|\chi u_{1,1} \|_{H^{s-1,1}_\theta}$ and $\|\chi' u_{1,1} \|_{H^{s-1,0}_\theta}$}. Using \eqref{b12}, we have
	\begin{equation}\label{BU}
		u_{1,1}(t)=\sum^{\infty}_{j=1} \frac{t^{j+1}}{j!} \int^1_0 \mathrm{e}^{\mathrm{i}t(2\rho-1)D}g_j(\rho)d\rho,
	\end{equation}
	where $g_j(\rho)$ is defined in \eqref{b11}. Due to \eqref{cr5} and \eqref{BU}, we get
	\begin{equation}\label{Opae}
		\begin{split}
			\| \chi u_{1,1} \|_{H^{s-1,1}_{\theta}} \lesssim & \big(\sum^{\infty}_{j=1} \frac{c_0^\theta \|t^{j+1} \chi \|_{L^2}+ \|t^{j+1} \chi \|_{H^\theta}}{j!} \big) \sup_{0\leq\rho\leq 1} \| g_j(\rho) \|_{H^{s-1,1}}
			\\
			\lesssim & \big(\sum^{\infty}_{j=1} \frac{c_0^\theta \|t^{j+1} \chi \|_{L^2}+ \|t^{j+1} \chi \|_{H^\theta}}{j!} \big) \cdot c_0 \cdot \sup_{0\leq\rho\leq 1} \| g_j(\rho) \|_{H^{s-1,0}}
			\\
			\leq & \big( \sum^\infty_{j=1} \frac{ (3c_0)^{j+\frac12} \|t^{j+1} \chi\|_{H^\theta}}{j!} +  \sum^\infty_{j=1} \frac{(3c_0)^{j+\frac12+\theta} \|t^{j+1} \chi\|_{L^2}}{j!} \big) \| F_1 \|_{H_0^{s-1,0}}
			\\
			\leq & C_0 \| F_1 \|_{H_0^{s-1,0}}.
		\end{split}
	\end{equation}
	Similarly, if we replace $\chi$ to $\chi'$, so we can also obtain
	\begin{equation}\label{op22}
		\begin{split}
			\| \chi' u_{1,1} \|_{H^{s-1,0}_{\theta}}
			\lesssim & \big(\sum^{\infty}_{j=1} \frac{c_0^\theta \|t^{j+1} \chi' \|_{L^2}+ \|t^{j+1} \chi' \|_{H^\theta}}{j!} \big) \sup_{0\leq\rho\leq 1} \| g_j(\rho) \|_{H^{s-1,0}}
			\\
			\leq & \big( \sum^\infty_{j=1} \frac{ (3c_0)^{j-\frac12} \|t^{j+1} \chi'\|_{H^\theta}}{j!} +  \sum^\infty_{j=1} \frac{(3c_0)^{j-\frac12+\theta} \|t^{j+1} \chi'\|_{L^2}}{j!} \big) \| F_1 \|_{H_0^{s-1,0}}
			\\
			\leq & C_0 \| F_1 \|_{H_0^{s-1,0}}.
		\end{split}
	\end{equation}
	\qquad \textbf{Step 2: Estimate for $\|\chi u_{1,2} \|_{H^{s-1,1}_\theta}$}. By \eqref{U12}, it yields
	\begin{equation*}
		u_{1,2}(t)=\sum^{\infty}_{j=1} \frac{t^j}{j!} \left( \mathrm{e}^{\mathrm{i}tD} f^+_j + \mathrm{e}^{-\mathrm{i}tD} f^-_j  \right)
		+R_{+}(t)+R_{-}(t).
	\end{equation*}
	Due to \eqref{B14} and \eqref{cr0}, we therefore derive
	\begin{equation}\label{op24}
		\begin{split}
			\left\| \chi(t) \sum^{\infty}_{j=1} \frac{t^j}{j!} \left( \mathrm{e}^{\mathrm{i}tD} f^+_j + \mathrm{e}^{-\mathrm{i}tD} f^-_j  \right) \right\|_{H^{s-1,1}_\theta} \leq \left(  \sum^{\infty}_{j=1} \frac{(3c_0)^{j-\frac12}\|t^j \chi \|_{H^\theta}}{j!} \right) \| F_1 \|_{H_0^{s-1,0}}.
		\end{split}
	\end{equation}
	Due to \eqref{Rf}, we can compute the estimate
	\begin{equation*}\label{op1}
		\begin{split}
			\widetilde{\chi R_+}(\tau,\xi)=-\frac{1}{4\pi {|\xi_1|}} \int^0_{-\infty} \frac{\widehat{\chi}(\tau-\lambda)-\widehat{\chi}(\tau-|\xi_1|)}{|\lambda|+|\xi_1|} \widetilde{F_{1,2}}({\lambda},\xi)d\lambda.
		\end{split}
	\end{equation*}
	By using Minkowski's inequality, we can bound $\|{\chi R_+}\|_{H^{s-1,1}_\theta}$ by 
	\begin{equation}\label{op3}
		\begin{split}
			\|{\chi R_+}\|_{H^{s-1,1}_\theta} \lesssim \int_{-\infty}^0  \left\| A(\lambda,\xi)(1+|\xi|)^{s-1} \widetilde{F_{1,2}}(\lambda,\xi) \right\|_{L^2_\xi}   d\lambda,
		\end{split}
	\end{equation}
	where
	\begin{equation*}
		\begin{split}
			A(\lambda,\xi)=\left\| (1+ | |\tau|-|\xi_1|| )^{\theta} \frac{\widehat{\chi}(\tau-\lambda)-\widehat{\chi}(\tau-|\xi_1|)}{|\lambda|+|\xi_1|} \right\|_{L^2_\tau}.
		\end{split}
	\end{equation*}
	To bound \eqref{op3}, we next provide an estimate for $A(\lambda, \xi)$. Define the set 
	$$\mathrm{U} = \left\{ \tau \in \mathbb{R}| |\tau-|\xi_1|| < 2( |\lambda| + |\xi_1| ) \right\}.$$ We then decompose $A(\lambda,\xi)$ as follows:
	\begin{equation}\label{op4}
		\begin{split}
			A(\lambda,\xi) \lesssim & \underbrace{ \left\|  \frac{\widehat{\chi}(\tau-\lambda)-\widehat{\chi}(\tau-|\xi_1|)}{|\lambda|+|\xi_1|} \right\|_{L^2_\tau} }_{\equiv A_1}
			+ \underbrace{ \left\|  | |\tau|-|\xi_1||^{\theta} \frac{\widehat{\chi}(\tau-\lambda)-\widehat{\chi}(\tau-|\xi_1|)}{|\lambda|+|\xi_1|} \right\|_{L^2_\tau(\mathrm{U})} }_{\equiv A_2}
			\\
			&	+ \underbrace{ \left\|  | |\tau|-|\xi_1||^{\theta} \frac{\widehat{\chi}(\tau-\lambda)-\widehat{\chi}(\tau-|\xi_1|)}{|\lambda|+|\xi_1|} \right\|_{L^2_\tau(\mathbb{R}\setminus\mathrm{U})} }_{\equiv A_3}.
		\end{split}
	\end{equation}
	For $\lambda<0$, then it follows that
	\begin{equation}\label{op5}
		\frac{\widehat{\chi}(\tau-\lambda)-\widehat{\chi}(\tau-|\xi_1|)}{|\lambda|+|\xi_1|}
		=\int^1_0 {\widehat{\chi}}'(\tau-|\xi_1|+\rho( |\lambda|+|\xi_1| )) d\rho.
	\end{equation}
	By Minkowski's inequality, it follows that
	\begin{equation}\label{op6}
		A_1 \lesssim \int^1_0 \| t\chi \|_{L^2(\mathbb{R})} d\rho = \|t\chi \|_{L^2(\mathbb{R})}.
	\end{equation}
	Due to
	\begin{equation*}
		||\tau|-|\xi_1|| \leq |\tau-|\xi_1|| \leq 2 | \tau-|\xi_1|+\rho( |\lambda|+|\xi_1| ) |, \quad \tau \in \mathbb{R}\backslash \mathrm{U},
	\end{equation*} 
	using \eqref{op5} again, we obtain
	\begin{equation}\label{op7}
		A_3 \lesssim \| t\chi \|_{H^\theta(\mathbb{R})} .
	\end{equation}
	On the other hand, for $\tau \in \mathrm{U}$, it yields
	\begin{equation*}\label{op70}
		|\tau-|\xi_1||\lesssim |\lambda|+|\xi_1|,  \quad	|\tau-\lambda|\lesssim |\lambda| + |\xi_1|.
	\end{equation*}
	Noting that $\theta \in (\frac12,1)$, and using \eqref{op70}, it follows
	\begin{equation}\label{op8}
		\begin{split}
			A_2  \lesssim 
			& \|  |\tau-|\xi_1||^{\theta}(|\lambda| + |\xi_1|)^{-1}  \widehat{\chi}(\tau-\lambda) \|_{L_\tau^2(\mathrm{U})}
			+\|  |\tau-|\xi_1||^{\theta}(|\lambda| + |\xi_1|)^{-1} \widehat{\chi}(\tau-|\xi_1|) \|_{L_\tau^2(\mathrm{U})}
			\\
			\lesssim & \|  (|\lambda|+|\xi_1|)^{\theta-1}  \widehat{\chi}(\tau-\lambda) \|_{L_\tau^2(\mathrm{U})}
			+ \|(\tau-|\xi_1|)^{\theta-1} \widehat{\chi}(\tau-|\xi_1|) \|_{L_\tau^2(\mathrm{U})}
			\\
			\lesssim & \|   |\tau-\lambda|^{\theta-1}  \widehat{\chi}(\tau-\lambda) \|_{L_\tau^2(\mathrm{U})}
			+ \|(\tau-|\xi_1|)^{\theta-1} \widehat{\chi}(\tau-|\xi_1|) \|_{L_\tau^2(\mathrm{U})}
			\\
			\lesssim & 2 \|\chi \|_{\dot{H}^{\theta-1}}.
		\end{split}
	\end{equation}
	Substituting \eqref{op6}, \eqref{op7}, and \eqref{op8} into \eqref{op4}, we obtain
	\begin{equation}\label{op10}
		A(\lambda,\xi) \lesssim  \| t\chi \|_{H^\theta(\mathbb{R})}+\|\chi \|_{\dot{H}^{\theta-1}}, \quad (\lambda,\xi)\in \mathrm{supp} \ \widehat{F_{1,2}}.
	\end{equation}
	Due to \eqref{op3}, \eqref{op10}, \eqref{lwh}, and by applying H\"older's inequality, we have
	\begin{equation}\label{op27}
		\begin{split}
			\| \chi R_+ \|_{H^{s-1,1}_\theta} \lesssim & {c_0^{\frac12}}(\| t\chi \|_{H^\theta(\mathbb{R})}+\|\chi \|_{\dot{H}^{\theta-1}}) 
			\left( \int_{-\infty}^0   \left\| (1+|\xi|)^{s-1} \widetilde{F_{1,2}}(\lambda,\xi) \right\|^2_{L^2_\xi} d\lambda \right)^{\frac12}
			\\
			\lesssim & c_0^{\frac12} (\| t\chi \|_{H^\theta(\mathbb{R})}+\|\chi \|_{\dot{H}^{\theta-1}})  \| F_1 \|_{H_0^{s-1,0}}.
		\end{split}
	\end{equation}
	Similarly, we can also derive
	\begin{equation}\label{op28}
		\begin{split}
			\| \chi R_- \|_{H^{s-1,1}_\theta} \lesssim & c_0^{\frac12} (\| t\chi \|_{H^\theta(\mathbb{R})}+\|\chi \|_{\dot{H}^{\theta-1}}) 
			\left( \int^{\infty}_0   \left\| (1+|\xi|)^{s-1} \widetilde{F_{1,2}}(\lambda,\xi) \right\|^2_{L^2_\xi} d\lambda \right)^{\frac12}
			\\
			\lesssim & c_0^{\frac12} (\| t\chi \|_{H^\theta(\mathbb{R})}+\|\chi \|_{\dot{H}^{\theta-1}})  \| F_1 \|_{H_0^{s-1,0}}.
		\end{split}
	\end{equation}
	Combining \eqref{op24}, \eqref{op27}, and \eqref{op28}, we have verified
	\begin{equation}\label{Op2}
		\begin{split}
			\| \chi u_{1,2}\|_{H^{s-1,1}_\theta} 
			\lesssim & \left(  \sum^{\infty}_{j=1} \frac{c_0^{j-\frac12}\|t^j \chi \|_{H^\theta}}{j!}   + c_0^{\frac12} (\| t\chi \|_{H^\theta(\mathbb{R})}+\|\chi \|_{\dot{H}^{\theta-1}})  \right) \| F_1 \|_{H_0^{s-1,0}}
			\\
			\leq & C_0  \| F_1 \|_{H_0^{s-1,0}}.
		\end{split}
	\end{equation}
	\textbf{Step 3: Estimate for $\|\chi' u_{1,2} \|_{H^{s-1,0}_\theta}$}. Using \eqref{B05} and \eqref{b02}, we can rewrite the expression as follows:
	\begin{equation*}\label{op13}
		\widehat{u_{1,2}}(t)=\frac{1}{4\pi |\xi_1|} \int_{\mathbb{R}} \left\{ \frac{\mathrm{e}^{\mathrm{i}t\tau}-\mathrm{e}^{\mathrm{i}t|\xi_1|}}{\tau-|\xi_1|} - \frac{\mathrm{e}^{\mathrm{i}t\tau}-\mathrm{e}^{-\mathrm{i}t|\xi_1|}}{\tau+|\xi_1|}  \right\}\widetilde{F_{1,2}}(\tau,\xi)d\tau.
	\end{equation*}
Therefore, it follows that
	\begin{equation}\label{op30}
		\begin{split}
			\widetilde{\chi u_{1,2}}=&\frac{1}{4\pi |\xi_1|} \int_{\mathbb{R}} \left\{ \frac{\widehat{\chi}(\tau-\lambda)-\widehat{\chi}(\tau-|\xi_1|)}{\lambda-|\xi_1|} - \frac{\widehat{\chi}(\tau-\lambda)-\widehat{\chi}(\tau+|\xi_1|)}{\lambda+|\xi_1|}  \right\}\widetilde{F_{1,2}}(\lambda,\xi)d\lambda
			\\
			=& \frac{1}{4\pi |\xi_1|} \int_{\mathbb{R}} \int^1_0 \left\{ \widehat{\chi}'(\tau-a)-\widehat{\chi}'(\tau-b) \right\} \widetilde{F_{1,2}}(\lambda,\xi) d\rho d\lambda.
		\end{split}
	\end{equation}
	where $a=|\xi_1|+\rho(\lambda-|\xi_1|)$ and $b=-|\xi_1|+\rho(\lambda+|\xi_1|)$. To bound $\chi u_{1,2}$, we divide it into two cases.

	\textit{Case 1: $|\tau|< 8|\xi_1|$.} Observe that, for $(\lambda,\xi)\in \mathrm{supp} \widetilde{F_{1,2}}$, $||\tau|-|\xi_1||,\ |\tau-a|,\ |\tau-b| \lesssim |\xi_1|$. Since $\theta\in(\frac12,1)$, it's clear for us to get
	\begin{equation}\label{op15}
		\begin{split}
			& (1+|\xi|)^{s-1} (1+ ||\tau|-|\xi_1||)^{\theta} | \widetilde{\chi u_{1,2}} | 
			\\
			\lesssim & \int_\mathbb{R} \int^1_0 \left( \frac{|\widetilde{\chi}'(\tau-a)|}{|\tau-a|^{1-\theta}} 
			+ \frac{|\widehat{\chi}'(\tau-b)|}{|\tau-b|^{1-\theta}}\right)
			(1+|\xi|)^{s-1} |\widetilde{F_{1,2}}(\lambda,\xi)| d\rho d\lambda .
		\end{split}
	\end{equation}

	\textit{Case 2: $|\tau|\geq 8|\xi_1|$.} In this case, we can rewrite \eqref{op30} by
	\begin{equation*}
		\begin{split}
			\widetilde{\chi u_{1,2}}(\tau,\xi)=&\frac{1}{2\pi } \int_{\mathbb{R}} \int^1_0 \int^1_0 \widehat{\chi}''(\tau-b+\sigma(b-a))(1-\rho) \widetilde{F_{1,2}}(\lambda,\xi)d\sigma d\rho d\lambda.
		\end{split}
	\end{equation*}
	Note that
	\begin{equation*}
		| \tau-b+\sigma(b-a) | \geq |\tau|-|b|-|b-a|\geq |\tau|-6|\xi_1| \gtrsim ||\tau|-|\xi_1||.
	\end{equation*}
	Therefore, we establish
	\begin{equation}\label{op17}
		\begin{split}
			& (1+|\xi|)^{s-1} (1+ ||\tau|-|\xi_1||)^{\theta} | \widetilde{\chi u_{1,2}} | 
			\\
			\lesssim & \int_\mathbb{R} \int^1_0 \int^1_0 (1+|\tau-b+\sigma(b-a)|^\theta)  
			|\widehat{\chi}''(\tau-b+\sigma(b-a))| (1+|\xi|)^{s-1} |\widetilde{F_{1,2}}(\lambda,\xi) | d\sigma d\rho d\lambda .
		\end{split}
	\end{equation}
	By applying \eqref{op15}, \eqref{op17}, and Minkowski's inequality, we can conclude that
	\begin{equation}\label{OPar}
		\begin{split}
			\| \chi u_{1,2} \|_{H^{s-1,0}_\theta} \lesssim & c_0^{\frac12} ( \|t\chi\|_{\dot{H}^{\theta-1}} + \|t^2 \chi \|_{H^\theta}) \|F_1\|_{H_0^{s-1,0}}.
		\end{split}
	\end{equation}
	As a result, replacing $\chi$ by $\chi'$ in \eqref{OPar}, we get
	\begin{equation}\label{OP00}
		\begin{split}
			\| \chi' u_{1,2} \|_{H^{s-1,0}_\theta} \lesssim & c_0^{\frac12} ( \|t\chi'\|_{\dot{H}^{\theta-1}} + \|t^2 \chi' \|_{H^\theta}) \|F_1\|_{H_0^{s-1,0}}.
		\end{split}
	\end{equation}
	Using \eqref{op22} and \eqref{OP00}, we have
	\begin{equation}\label{Op3}
		\begin{split}
			\| \chi' u_{1} \|_{H^{s-1,0}_\theta} 
			\leq & c_0^{\frac12} ( \|t\chi'\|_{\dot{H}^{\theta-1}} + \|t^2 \chi' \|_{H^\theta}) \|F_1\|_{H_0^{s-1,0}}
			\\
			& \quad + \big( \sum^\infty_{j=1} \frac{c_0^{j-\frac12} \|t^{j+1} \chi'\|_{H^\theta}}{j!} +  \sum^\infty_{j=1} \frac{c_0^{j-\frac12+\theta} \|t^{j+1} \chi'\|_{L^2}}{j!} \big) \|F_1\|_{H_0^{s-1,0}}
			\\
			\leq  & C_0 \|F_1\|_{H_0^{s-1,0}}.
		\end{split}
	\end{equation}
	
	\textbf{Step 4: Estimate for $\|\chi \partial_t u_{1} \|_{H^{s-1,0}_\theta}$}. By \eqref{hj}, \eqref{hj1}, \eqref{op27}, and \eqref{op28}, a straightforward modification of the argument used to estimate $\|\chi u_{1,2} \|_{H^{s-1,1}_\theta}$ shows that
	\begin{equation}\label{Op4}
		\begin{split}
			\| \chi \partial_t u_{1} \|_{H^{s-1,0}_\theta} \lesssim & \big(  \sum^{\infty}_{j=1} \frac{(3c_0)^{j-\frac12}\|t^j \chi \|_{H^\theta}}{j!} \big) \| F_1 \|_{H_0^{s-1,0}}
			+ c_0^{\frac12} ( \|\chi\|_{\dot{H}^{\theta-1}} + \|t \chi \|_{H^\theta}) \|F_1\|_{H_0^{s-1,0}}
			\\
			\leq & C_0  \|F_1\|_{H_0^{s-1,0}}.
		\end{split}
	\end{equation}
Summarizing estimates \eqref{Opae}, \eqref{op22}, \eqref{Op2}, \eqref{Op3}, and \eqref{Op4}, we obtain \eqref{op0}. Thus, the proof of Lemma \ref{oE} is complete.
\end{proof}
Based on the results above, we are now ready to provide a proof of Proposition \ref{nE}.
\subsection{Proof of Proposition \ref{nE}.}
\begin{proof}[proof of Proposition \ref{nE}]
	By Lemma \ref{cr0} and \ref{cr2}, we can see
	\begin{equation}\label{pt0}
		|\chi(t)u_0|_{H^{s-1,1}_\theta} \lesssim C( \|f\|_{H^{s-1,1}}+ \|g\|_{H^{s-1,0}} ),
	\end{equation}
	where $C=\| \chi \|_{H^\theta}+\| t \chi \|_{H^\theta} +\| \chi' \|_{H^\theta}+\| t\chi' \|_{H^\theta}$. Using \eqref{defu0}, we can verify that
	\begin{equation*}\label{pt1}
		\mathrm{supp} \widetilde{F_1} \subseteq 4\sqrt{2} \mathcal{N}, \quad \mathrm{supp} \widetilde{F_2} \subseteq \mathbb{R}^{1+n} \backslash \mathcal{N}.
	\end{equation*}
	Define 
	\begin{equation*}
		\chi_T(t)=\chi(\frac{t}{T}),\qquad t\in\mathbb{R}.
	\end{equation*}
	We therefore obtain
	\begin{equation}\label{pt2}
		||\tau|-|\xi_1|| \leq 2\sqrt{2} T^{-\frac12}, \quad \mathrm{for} \ (\tau,\xi)\in \mathrm{supp} \widetilde{F_1}.
	\end{equation}
	Substituting \eqref{pt2} into \eqref{lwh}, we can choose
	\begin{equation}\label{pt4}
		c_0=2+2\sqrt{2} T^{-\frac12}.
	\end{equation}
	By \eqref{op0} in Lemma \ref{oE}, it follows that
	\begin{equation}\label{pt5}
		\begin{split}
			| \chi_{T} u_1 |_{H^{s-1,1}_\theta} \leq \bar{C}_T \|F_1\|_{H^{s-1,0}_0},
		\end{split}
	\end{equation}
	where
	\begin{equation}\label{CTb}
		\begin{split}
			\bar{C}_T \simeq \ & (3c_0)^{\frac12} \left( \|\chi_T\|_{\dot{H}^{\theta-1}}+\|t\chi_T\|_{H^{\theta}}+\|t\chi_T'\|_{\dot{H}^{\theta-1}}+\|t^2\chi_T'\|_{H^{\theta}}  \right)
			\\
			& +\sum^\infty_{j=1} \big( \frac{(3c_0)^{j+\frac12}}{j!} \| t^{j+1}\chi_T \|_{H^\theta}+ \frac{(3c_0)^{j+\frac12+\theta}}{j!} \| t^{j+1}\chi_T \|_{L^2} 
			\\
			& \quad +\frac{(3c_0)^{j-\frac12}}{j!} \| t^{j+1}\chi_T' \|_{H^\theta}
			+ \frac{(3c_0)^{j-\frac12+\theta}}{j!} \| t^{j+1}\chi'_T \|_{L^2} 
			+ \frac{(3c_0)^{j-\frac12}}{j!} \| t^{j}\chi_T \|_{H^\theta} \big).
		\end{split}
	\end{equation}
	By \eqref{pt2} and \eqref{pt4}, we have
	\begin{equation}\label{pt7}
		\begin{split}
			\|F_1\|_{H^{s-1,0}_0} \leq c_0^{1-(\theta+\varepsilon)} \|F_1\|_{H^{s-1,0}_{\theta+\varepsilon-1}}\leq (3c_0)^{1-\theta}\|F_1\|_{H^{s-1,0}_{\theta+\varepsilon-1}}.
		\end{split}
	\end{equation}
	Combining \eqref{pt5} and \eqref{pt7} yields
	\begin{equation}\label{ptqw}
		\begin{split}
			| \chi_{T} u_1 |_{H^{s-1,1}_\theta} \leq {C}_T \|F_1\|_{H^{s-1,0}_{\theta {+\varepsilon}-1}},
		\end{split}
	\end{equation}
	where 
	\begin{equation}\label{pt9}
		\begin{split}
			C_T=\bar{C}_T(3c_0)^{1-\theta}.
		\end{split}
	\end{equation}
	Due to $\theta \in (\frac12,1)$ and $0<T\leq 1$, shows that
	\begin{equation}\label{pt10}
		\begin{split}
			\| \chi_T\|_{H^\theta} \lesssim & T^{\frac12-\theta} \|\chi\|_{H^\theta},
			\\
			\| \chi_T\|_{\dot{H}^{\theta-1}} = & T^{\frac32-\theta} \|\chi\|_{\dot{H}^{\theta-1}}.
		\end{split}
	\end{equation}
Inserting \eqref{pt10} into \eqref{pt9} and using \eqref{CTb} yields
	\begin{equation}\label{pt12}
		\begin{split}
			{C}_T \lesssim \ & (3c_0 T)^{\frac32-\theta} \left( \|\chi\|_{\dot{H}^{\theta-1}}+\|t\chi\|_{H^{\theta}}+\|t\chi'\|_{\dot{H}^{\theta-1}}+\|t^2 \chi'\|_{H^{\theta}}  \right)
			\\
			& +\sum^\infty_{j=1} \big( \frac{(3c_0T)^{j+\frac32-\theta}}{j!} \| t^{j+1}\chi \|_{H^\theta}+ \frac{(3c_0T)^{j+\frac32}}{j!} \| t^{j+1}\chi\|_{L^2} 
			\\
			& \quad +\frac{(3c_0T)^{j+\frac12-\theta}}{j!} \| t^{j+1}\chi' \|_{H^\theta}
			+ \frac{(3c_0T)^{j+\frac12}}{j!} \| t^{j+1}\chi' \|_{L^2} 
			+ \frac{(3c_0T)^{j+\frac12-\theta}}{j!} \| t^{j}\chi \|_{H^\theta} \big).
		\end{split}
	\end{equation}
	Observing \eqref{pt4} and noting that $0<T<1$, we derive that
	\begin{equation}\label{pt14}
		3c_0 \leq (6+6\sqrt{2})T^{-\frac12} .
	\end{equation}
	Combining \eqref{pt12} and \eqref{pt14}, we conclude that
	\begin{equation}\label{pt15}
		C_T \lesssim C_* T^{\frac14},
	\end{equation}
	where
	\begin{equation}\label{pt17}
		\begin{split}
			{C}_* =\ & (6+6\sqrt{2})(\|\chi\|_{\dot{H}^{\theta-1}}+\|t\chi\|_{H^{\theta}}+\|t\chi'\|_{\dot{H}^{\theta-1}}+\|t^2 \chi'\|_{H^{\theta}} )
			\\
			& +\sum^\infty_{j=1} \frac{(6+6\sqrt{2})^{j+2}}{j!} \left(  \| t^{j+1}\chi \|_{H^\theta}+ \| t^{j+1}\chi'\|_{H^\theta} + \| t^{j}\chi \|_{H^\theta}  \right)  .
		\end{split}
	\end{equation}
	For $\chi$ satisfying \eqref{chi}, this implies that $C_*$ is a bounded constant depending only on $\chi$ and $\theta$. Summing up our results \eqref{ptqw}, \eqref{pt15}, and \eqref{pt17}, we obtain
	\begin{equation}\label{pt18}
		\begin{split}
			| \chi_{T} u_1 |_{H^{s-1,1}_\theta} \leq {C}_* T^{\frac14} \|F_1\|_{H^{s-1,0}_{\theta+\varepsilon-1}} \textcolor{red}{.}
		\end{split}
	\end{equation}
	On the other hand, {if $(\tau,\xi)\in \mathrm{supp}\widetilde{F_2}$, then}
	\begin{equation*}
	||\tau|-|\xi_1||> T^{-\frac12}.
	\end{equation*}
Following the same steps as in the proof of \eqref{up00}, we thus obtain 
	\begin{equation}\label{pt20}
		\begin{split}
			| \chi_{T} u_2 |_{H^{s-1,1}_\theta} \lesssim T^{\frac{\varepsilon}{2}} \|F_2\|_{H^{s-1,0}_{\theta+\varepsilon-1}},
		\end{split}
	\end{equation}
where the term $\chi_{T}$ is handled analogously to the case of $u_1$.	Adding \eqref{pt0}, \eqref{pt18} and \eqref{pt20}, we have proved \eqref{none1}. Due to the representation of solutions for linear waves, the function \(u\) defined in \eqref{defu} satisfies the Cauchy problem \eqref{linearw} for \((t,\by) \in [0,T] \times \mathbb{R}^n\). By standard energy estimates, $u$ is the unique solution of \eqref{linearw}. Therefore, the proof of Proposition \ref{nE} is complete.
\end{proof}

\section*{Acknowledgments} 
The author gratefully acknowledges the referee for their insightful comments and valuable suggestions, which have significantly improved the quality of the manuscript. The author also extends sincere thanks to Yi Zhou and Yuan Cai for their helpful discussions and feedback. This work is supported by the Natural Science Foundation of 
Hunan Province(Grant No. 2025JJ40003), and the Fundamental Research Funds for the Central Universities(Grant No. 531118010867).

\section*{Conflicts of interest and Data Availability Statements}
The author declared that this work does not have any conflicts of interest. The author also confirms that the data supporting the findings of this study are available within the article.

\end{document}